\documentclass[11pt]{article}
 \usepackage{bbm}
  \usepackage{amssymb}
\usepackage{amssymb, amsthm, amsmath, amscd}
\usepackage[usenames,dvipsnames]{color}
\usepackage{hyperref}
\hypersetup{hidelinks, colorlinks=true, allcolors=blue, pdfstartview=Fit, breaklinks=true}

\newtheorem{thm}{Theorem}[section]
\newtheorem{lem}[thm]{Lemma}
\newtheorem{prop}[thm]{Proposition}
\newtheorem{cor}[thm]{Corollary}
\newtheorem{NN}[thm]{}

\theoremstyle{definition}\newtheorem{df}[thm]{Definition}
\theoremstyle{definition}\newtheorem{rem}[thm]{Remark}
\theoremstyle{definition}\newtheorem{exm}[thm]{Example}

\renewcommand{\phi}{\varphi}

\newcommand{\N}{\mathbb{N}}

\newcommand{\Q}{\mathbb{Q}}
\newcommand{\R}{\mathbb{R}}
\newcommand{\C}{\mathbb{C}}

\newcommand{\Aff}{\operatorname{Aff}}

\newcommand{\id}{\operatorname{id}}

\newcommand{\hm}{homomorphism}
\newcommand{\dt}{\delta}
\newcommand{\ep}{\epsilon}
\newcommand{\la}{\langle}
\newcommand{\ra}{\rangle}
\newcommand{\andeqn}{\,\,\,{\rm and}\,\,\,}
\newcommand{\rforal}{\,\,\,{\rm for\,\,\,all}\,\,\,}
\newcommand{\CA}{$C^*$-algebra}
\newcommand{\SCA}{$C^*$-subalgebra}

\newcommand{\af}{{\alpha}}
\newcommand{\bt}{{\beta}}

\newcommand{\beq}{\begin{eqnarray}}
\newcommand{\eneq}{\end{eqnarray}}
\newcommand{\tforal}{\,\,\,\text{for\,\,\,all}\,\,\,}
\newcommand{\tand}{\,\,\,\text{and}\,\,\,}
\newcommand{\Qw}{\overline{QT(A)}^w}
\newcommand{\LAff}{{\rm LAff}}
\newcommand{\cuapprox}{\stackrel{\approx}{\sim}}
\newcommand{\Wlog}{Without loss of generality}

\newcommand{\simle}{\stackrel{\sim}{<}}

\title{Double duals and Hilbert modules
}
\author{Huaxin Lin}
\begin{document}

\maketitle

\begin{abstract}
Let $A$ be a  \CA, $H$ be a 
Hilbert $A$-module and $K(H)$ be the closure of the set of finite rank module maps.
We show that the $W^*$-algebra of all  bounded 
$A^{**}$-module maps on the smallest self-dual  Hilbert $A^{**}$-module containing $H$
is isomorphic to  $K(H)^{**}$ as $W^*$-algebras.
We also show that the unit ball of $H$ is closed in $H^\sharp,$ the dual of $H,$ in an $A$-weak topology of $H^\sharp$
as well as  dense in the unit ball of $H^\sharp$ in a weak*-topology 
and some versions of  Kaplansky density theorem for Hilbert $C^*$-modules.

\end{abstract}

\section{Introduction}

Hilbert $C^*$-modules  as a generalization of Hilbert spaces were first introduced by
I. Kaplansky (\cite{Kap})  in 1953  in special cases and later by  W. Paschke  (\cite{Pa}) for general \CA s in 1973.
Hilbert $C^*$-modules are crucial to Kasparov's formulation of $KK$-theory (\cite{Ka}).  
Early applications also include $C^*$-algebraic quantum group theory (see \cite{BS}). Later in the study of Cuntz -semigroups in connection with the classification of amenable \CA s, Hilbert $C^*$-modules play an important role (see, for example, 
\cite{BL}, \cite{BC}, \cite{CEI}, and \cite{ORT}).

Let $A$ be a \CA. Unlike Hilbert spaces, 
bounded module maps on a Hilbert $A$-module $H$
may not have adjoints and  the dual module $H^\sharp,$ i.e., the  Banach $A$-module of all bounded module 
maps from $H$ to $A$ may not be identified as elements in $H.$  Moreover, the \CA\, $L(H)$ of all bounded module maps 
with adjoints may not be a $W^*$-algebras.   If $H_0\subset H$ is a Hilbert $A$-submodule,
a bounded module map $\phi: H_0\to A$ may not be extended to a bounded module map from
$H$ to $A.$ In general, one should not expect  that  $H$ to be decomposed to an orthogonal direct 
sum of $H_0$ and its orthogonal complement. In fact, 
$H_0$  may not even have  an  orthogonal complement.
Study of these phenomena may be found, for eaxmples, in  \cite{Linbd}, \cite{Lininj} and more recently 
in \cite{BL}. 

However, Paschke found (see \cite{Pa}) out that, if $A$ is a $W^*$-algebra, then 
 the dual module $H^\sharp$ of a Hilbert $A$-module $H$ can be made into
a Hilbert $A$-module in a natural way which extends $H,$ and $H^\sharp$ is a self-dual Hilbert $A$-module.
Even if $A$ is not a $W^*$-algebra, one can extend $H$ into an $A^{**}$-module $H\hspace{-0.03in}\bullet\hspace{-0.03in} A^{**}$
naturally. Then its dual $H^\sim:=(H\hspace{-0.03in}\bullet\hspace{-0.03in} A^{**})^\sharp$ becomes a self-dual Hilbert 
$A^{**}$-module containing $H.$    In fact, $H^\sim$ is the smallest 
self-dual Hilbert $A^{**}$-module containing $H$ as a Hilbert $A$-submodule (see Proposition 
\ref{Psmall2}).
Paschke showed  the Banach algebra of all bounded module maps on 
$H^\sim$ becomes a $W^*$-algebra.   

For a Hilbert $A$-module $H,$ the rank one module maps 
are the module maps $T$ of the form $T(h)=x\la y, h\ra$ for all $h\in H$ (and fixed $x, y\in H,$
where $\la \cdot, \cdot\ra$ is the $A$-valued inner product).
Denote by $F(H)$ the linear span of rank one module maps and denote by
$K(H)$ the norm closure of $F(H).$ $K(H)$ is a \CA\, and  an important algebra related 
to the Hilbert module $H.$ 
It was proved by Kasparov (see Theorem 1 of \cite{Ka}) that  the \CA\, 
$L(H)$ may be identified with $M(K(H)),$ the multiplier algebra of $K(H),$ and  it was proved in 
\cite{Linbd} that the Banach algebra of all bounded module maps on $H$ is identified 
with the left multipliers of $K(H)$ (all Hilbert $A$-modules considered in this paper are 
right $A$-modules). 
Over the decades, we eventually realized 
that  it is rather convenient to work in $B(H^\sim)$ in many occasions as we study module maps on  a Hilbert module $H.$   It is not difficult to establish a natural normal \hm\, $\Psi: K(H)^{**}\to B(H^\sim)$
which extends beyond $M(K(H))$ and $LM(K(H)).$  
It remained unknown for many years whether 
$\Psi$ is an isomorphism.   
The original motivation  of this paper is to show that indeed $\Psi$ is 
an isomorphism between $W^*$-algebras $K(H)^{**}$ and $B(H^\sim).$

As we study the relation among Hilbert modules $H,$ $H\hspace{-0.03in}\bullet\hspace{-0.03in} A^{**}$ and $H^\sim,$
naturally, we ask how dense  of $H$ in $H\hspace{-0.03in}\bullet\hspace{-0.03in} A^{**}$ and in $H^\sim$  is?
Since $H^\sim=(H\hspace{-0.03in}\bullet\hspace{-0.03in} A^{**})^\sharp,$ the dual of $H\hspace{-0.03in}\bullet\hspace{-0.03in} A^{**},$ one may also ask the density of $H$ in $H^\sharp$ in general. 

We first note that it was shown (Theorem 6.1 of \cite{BL}) that 
$H$ is dense in $H^\sharp$ in an $A$-weak topology. More precisely,  for any 
$\xi\in H^\sharp,$ there is a net $\{x_\af\}$ in $H$ with $\|x_\af\|\le \|\xi\|$ for all $\af$
such that $\lim_\af \|\xi(x)-\la x_\af, x\ra \|=0$ for all $x\in H.$   However, we show here 
that the unit ball of $H$ is closed in $H^\sharp$ in the topology
where $x_\af\to \xi$ if and only if 
$\lim_\af \|\la \xi-x_\af, \zeta\ra\|=0$ for all $\zeta\in H^\sharp,$  and where the inner product 
is extended to  $H^\sim.$

On the other hand, 
it is easy to see that for any $\xi\in H\hspace{-0.03in}\bullet\hspace{-0.03in} A^{**}$ 
there is a net $\{x_\lambda\}$ in $H$ such that $\lim_\lambda \pi_U(\la x_\lambda, y\ra)(v)=
\pi_U(\la \xi, y\ra)(v)$ for all $y\in H\hspace{-0.03in}\bullet\hspace{-0.03in} A^{**}$ and $v\in H_U,$  where
$H_U$ is the Hilbert space corresponding to the universal representation $\pi_U$ of $A.$
%A more important question 
To be a more useful approximation, one may ask
whether the net can be chosen to be bounded
(by $\|\xi\|$).  We will present a 
Kaplansky style density theorem. Perhaps more interesting question is 
how dense $H$ is in $H^\sim=(H\hspace{-0.03in}\bullet\hspace{-0.03in} A^{**})^\sharp.$ 
Since $H^\sim$ is the dual of $H\hspace{-0.03in}\bullet\hspace{-0.03in} A^{**},$
it is relatively easy to show that, for any $\zeta \in H^\sim,$ there 
is a net $\{z_\af\}$ in $H$ such that 
$\lim_\lambda f(\la z_\af, y\ra)=f(\la \zeta, y\ra)$ for all $y\in H\hspace{-0.03in}\bullet\hspace{-0.03in} A^{**}$ and $f\in A^*.$
It is more challenging to show that $y$ can be replaced by any element in $H^\sim=(H\hspace{-0.03in}\bullet\hspace{-0.03in} A^{**})^\sharp.$
We show that the unit ball of $H$ 
is actually dense in the unit ball of $H^\sim$ in the weak*-topology (as $H^\sim$ is 
a conjugate space), another Kaplansky style density theorem. 
In fact, we show a stronger density theorem  that, for any $\xi\in H^\sim,$ there is a net $\{x_\af\}$ in $H$
with $\|x_\af\|\le \|\xi\|$ such that
\beq
\lim_\af  f(\la \xi-x_\af,\xi-x_\af\ra)=0\rforal f\in A^*.
\eneq

\section{Self-duals}

\begin{df}\label{A11}
Let $A$ be a \CA. Denote by $\tilde A$ the minimum unitization of $A.$
We will use the following convention: If $A$ is a \SCA\, of a unital \CA\, $B,$  we write 
$1_{\tilde A}=1_B,$ if either $A$ is unital and $1_A=1_{\tilde A}=1_B,$ or 
$A^\perp=\{b\in B: ba=ab=0\}=\{0\},$
%$bA=0,$ or $Ab=0$ implies that $b=0,$ 
and we  unitize $A$ by
adjoint $1_B$ to form $\tilde A\subset B.$
\end{df}

\begin{df}\label{MMSOT}
Let $X$ be a Hilbert space and $B(X)$ be the \CA\, of all bounded linear operators 
on $X.$ Suppose that $A\subset B(X).$ Then $\overline{A}^{SOT}$ is the  closure of $A$ 
in the strong operator topology.  
Note that if $\{e_\af\}$ is an approximate identity for $A,$ 
then $e_\af\nearrow 1_{M},$ i.e., $e_\af$ increasingly converges to the identity of $M=\overline{A}^{SOT}$
in the strong operator topology as well as in the weak*-topology (of $M$). 
In particular,  we may write $1_{\tilde A}=1_M.$  

This works particularly for the pair $A$ and $A^{**}$ (where $X$ is $H_u,$ 
the Hilbert space corresponding to the universal representation of $A$). 

In general, if $M$ is a $W^*$-algebra, denote by  $M_*$ the pre-dual of $M.$
\end{df}

\begin{df}\label{D101}
Let $A$ be a \CA.
%\, and $H$ be a Hilbert $A$-module (i
In this paper,  we use the formal definition of Hilbert modules in \cite{Pa}
and consider only right $A$-modules. Recall that  a linear space  $H$ is a pre-Hilbert module
if it is also a right $A$-module
with an inner product $H\times H\to A$ satisfying the following properties:
for any $x, y, z\in H,$ $a\in A$ and $\lambda\in \C,$

(i) $\la x, \lambda y+z\ra=\la x, y\ra+\lambda \la x, y\ra;$

(ii) $\la x, y a\ra=\la x, y\ra a;$

(iii) $\la x, y\ra^*=\la y, x\ra;$

(iv) $\la x, x\ra\ge 0;$ if $\la x, x\ra=0,$ then $x=0.$
 
 Define $\| x\|=\|\la x, x\ra\|^{1/2}$ for $x\in H.$ Then $H$ becomes a normed space. 
 $H$ is a Hilbert $A$-module if $H$ is complete with this norm. 
% and use definition  in \cite{Pa} 
%for Hilbert $A$-module).

Denote by $H^\sharp$ the Banach space of all bounded module maps 
from $H$ into $A.$  A Hilbert $A$-module is said to be self-dual, if, for 
every $f\in H^\sharp,$  there is $x\in H$ such that 
$$
f(y)=\la x, y\ra \rforal y\in H.
$$

Denote by $B(H)$ the Banach algebra of all bounded module maps 
from $H$ into itself, and by $L(H),$ the \CA\, of all those bounded module maps
$T$ with an adjoint $T^*$ in $L(H),$ defined by
$$
\la T(x), y\ra=\la x, T^*(y)\ra \rforal x, y\in H.
$$

Let $F(H)$ be the algebra of all finite rank module maps,
i.e., the linear span of all
bounded module maps 
of the form:
$\theta_{x, y}: H\to H,$ defined by
$$
\theta_{x, y}(\xi)=x\la y, \xi\ra
$$
for all $\xi\in H$ and $x, y\in H.$
Denote by $K(H)$ the \CA\, of all bounded module maps which 
are (uniform) norm closure of $F(H).$

By Theorem 1 of \cite{Ka}, we identify $L(H)$ with $M(K(H))$ the multiplier algebra of $K(H)$
and, by Theorem 1.5 of \cite{Linbd},  $B(H)$ with $LM(K(H))$ the Banach algebra of left multipliers of $K(H)$ ( in $K(H)^{**}$).  If $H$ is self-dual, then $B(H)=L(H).$

We refer to \cite{Ka}, \cite{Pa}, \cite{Linbd}  and \cite{Lininj} for common terminologies related to Hilbert $C^*$-modules.
\end{df}

\begin{df}\label{DHsim}
Let $A$ be a \CA\, and $H$   a Hilbert $A$-module. 
Let us give the definition of self-dual Hilbert $A^{**}$-module $H^\sim$ (see Definition 1.3 of \cite{Linbd}). 

We may view $H$ as a Hilbert 
$\tilde A$-module.
% where $\tilde A$ is the minimal unitization of $A.$ 
Let $B$ be a unital \CA\, containing $A$ and $1_{\tilde A}=1_B$ 
(see the convention in \ref{A11}).
%If $A$ has a unit, we assume that $1_B=1_A.$ 
%Otherwise, 
%we may identify $\tilde A$ with the \SCA\, of $B$ generated by $1_B$ and $A.$
The algebraical tensor product $H\otimes B$  becomes a right $B$-module 
if we set $(h\otimes a) \cdot b= h\otimes ab$ for any $h\in H$ and $a, b\in B.$ 
Define $\la -, -\ra: H\otimes B\times H\otimes B\to B$ by 
$$
\la  \sum_i h_i\otimes a_i,\, \sum_j x_j, \otimes b_j\ra =\sum_{i,j} a_i \la h_i, x_j\ra b_j,
$$
and $N=\{z\in H\otimes A^{**}: \la z, z\ra=0\}.$ 
Then $(H\otimes B)/N$ becomes a per-Hilbert $B$-module (see 4 of \cite{Pa}
--but exchange $B$ with $A$).
Denote by $H\hspace{-0.03in}\bullet\hspace{-0.03in} B:=(H\otimes B)/N)^-$ (the completion) the Hilbert $B$-module. 

We are particularly interested in the case that $B=A^{**}.$
We view $\tilde A$ as a \SCA\, of $A^{**}.$ Then $H^\sim:=(H\hspace{-0.04in}\bullet \hspace{-0.04in}A^{**})^\sharp$
is a self-dual Hilbert $A^{**}$-module. 

Note  that $\tilde A$ is ultraweakly dense in $A^{**}$
(since $A$ is). 
 By applying Theorem 4.2 of \cite{Pa} 
to the pair $A^{**}$ (as $A$ in Theorem 4.2 of \cite{Pa}) and $\tilde A$ (as $B$ in Theorem 
4.2 of \cite{Pa}, see also the remark right after the proof of Theorem 4.2 of \cite{Pa}), we obtain
% a self-dual Hilbert 
%$A^{**}$-module 
an isometric (surjective) isomorphism  $\iota: H^\sim:=(H\hspace{-0.03in}\bullet\hspace{-0.03in} A^{**})^\sharp\to B(H, A^{**}),$
the Banach space of all bounded $A$-module maps from $H$ to $A^{**}$ (as written $M(H, A^{**})$ 
in Theorem 4.2 of \cite{Pa}).
  
Let $x\in H$ and $b\in B.$ 
Then 
$$
\|(x\otimes b)/N\|^2=\|b^*\la x, x\ra b\| \le \|x\|^2 \|b^*b\|.
$$
Hence 
\beq
 \|(x\otimes b)/N\|\le \|x\|\|b\|.
\eneq
In what follows, for $x\in H$ and $b\in B,$
we write $x\hspace{-0.03in}\bullet\hspace{-0.03in} b:= (x\otimes b)/N.$

\iffalse
In general, for a pair of \CA s $A\subset B,$  where $B$ is unital, if $H$ is a Hilbert 
$A$-module, then exactly the same construction gives us 
a Hilbert $B$-module which we may denote it by
$H\hspace{-0.03in}\bullet\hspace{-0.03in} B.$ 
\fi

In general, if $E$ is a self-dual Hilbert module, then $B(E)=L(E)$ (see 3.5 of \cite{Pa}).
If in addition,  $A$ is a $W^*$-algebra, $B(E)$ is also a 
 $W^*$ -algebra (see Proposition 3.11 of \cite{Pa}). 
 
 Let us recall the description of the pre-dual of $B(E)$ in this case. 
 Denote by $E_\sim$ the linear space $E$ with the ``twisted" scalar 
 multiplication (i.e., $\lambda x=\bar \lambda x$ for $x\in E$ and $\lambda\in \C$)
 and consider $E\otimes E_\sim \otimes A_*$ with the greatest cross-norm,
 where $A_*$ is the usual pre-dual of $W^*$-algebra $A.$
 For each $T\in B(E),$ define a linear functional $\check{T}$ on $E_\sim  \otimes E\otimes A_*$
 by
 \beq
 \check{T}(\sum_{j=1}^n x_j\otimes y_j \otimes g_j)=\sum_{j=1}^ng_j(\la T(x_j), y_j\ra)
 \eneq
 for $x_j, y_j\in E$ and $g_j\in A_*,$ $1\le j\le n.$
 The map $T\to \Check{T}$ is a linear isometry of $B(E)=L(E)$ into $(E_\sim \otimes E \otimes A_*)^*.$
 It was shown (Proposition 3.10 of \cite{Pa}) that $B(E{\Check{)}}$  is weak*-closed in $E_\sim\otimes E\otimes A_*.$
 A bounded net $\{T_\af\}$ in $B(E)$ converges to $T\in B(E)$ in the weak*-topology if and only if 
 $f(\la T_\af(x), y\ra)\to f(\la T(x), y\ra)$
 for all $x, y\in E$ and $f\in A_*$ (Remark 3.9 and Proposition 3.10 of \cite{Pa}). 
  In particular, $B(H^\sim)$ is a $W^*$-algebra.

\end{df}

\begin{df}\label{Dmap1}
Keep the notation in \ref{DHsim}. 
Then $\iota: H\to H\hspace{-0.03in}\bullet\hspace{-0.03in} B$ defined by $x\to x\otimes 1$ is an injective map.
 Note that 
\beq
&&\la (x \cdot b)\otimes 1-x\otimes b, (x \cdot b)\otimes 1-x\otimes b\ra\\
&&=\la x \cdot b, x \cdot b\ra  -\la x \cdot  b, x\ra b-b^*\la x, x\cdot b\ra+b^*\la x, x\ra b\\
&&=b^*\la x, x\ra b-b^*\la x, x\ra b-b^*\la x, x\ra b+b^*\la x, x\ra b=0. 
\eneq
Hence  $\iota(x\cdot b)=x\otimes b/N$ for all $b\in B.$
In  the case $B=A^{**},$ 
we then extend $\iota$ from $H^\sharp$ to $(H\hspace{-0.03in}\bullet\hspace{-0.03in} A^{**})^\sharp$
by
$$
\iota(f)(x\otimes b)=f(x)b\rforal x\in H \andeqn b\in A^{**}
$$
and $f\in H^\sharp.$  Note that the map is a module map from $H^\sharp$ to  $(H^\sim)^\sharp$
 (which is conjugate module isomorphic to $H^\sim$).

From now on, we may view $H$ as a submodule of $H^\sim$ and, sometime, omit the map
$\iota.$
\end{df}

The following is a convenient  and easy fact that $H\hspace{-0.03in}\bullet\hspace{-0.03in}B$
is smallest Hilbert $B$-module containing $H$  as a Hilbert $A$-module.

\begin{prop}\label{Psmall}
Let $A$  and $B$ be a pair of \CA s such that $A\subset B,$ $B$ is unital and 
$1_{\tilde A}=1_B.$
Suppose that $H$ is a Hilbert $A$-module, $H_1$ is a Hilbert $B$-module
and there is an embedding $\iota: H\to H_1$ as Hilbert modules, 
i.e., $\iota$ is a linear and  $A$-module map such that
\beq
\la \iota(x), \iota(y)\ra=\la x, y\ra\tforal x,y\in H.
\eneq
Then there is a unique $B$-module embedding 
$\tilde \iota: H\hspace{-0.03in}\bullet\hspace{-0.03in}B\to H_1$
(such that
\beq
&&\tilde \iota(x\hspace{-0.03in}\bullet\hspace{-0.03in} b)=\iota(x)
%\hspace{-0.03in}\bullet\hspace{-0.03in} 
b\tforal x\in H\andeqn b\in B\tand\\ 
&&\la\tilde  \iota(\xi),\tilde \iota(\zeta)\ra=\la \xi , \zeta\ra \tforal \xi, \zeta\in H\hspace{-0.03in}\bullet\hspace{-0.03in}B).
\eneq
\end{prop}

\begin{proof}
For any $\sum_{i=1}^nx_i\hspace{-0.03in}\bullet\hspace{-0.03in}a_i,$
where $x_i\in H$ and $a_i\in B$ ($1\le i\le n$),
define 
\beq
\tilde \iota(\xi)=\sum_{i=1}^n \iota(x_i)a_i.
\eneq 
Then, for $\zeta=\sum_{i=1}^n y_i\hspace{-0.03in}\bullet\hspace{-0.03in} b_i,$ 
\beq
\la \tilde \iota(\xi), \tilde \iota(\zeta)\ra =\sum_{i,j}^na_i^*\la x_i, y_j\ra b_j=\la \xi, \zeta\ra.
\eneq
In particular,
\beq
\|\tilde \iota(\xi), \tilde \iota(\xi)\ra\| =\|\sum_{i,j}^na_i^*\la x_i, x_j\ra b_j\|=\|\xi\|^2.
\eneq 
Therefore $\|\tilde \iota\|\le 1$ on $(H\otimes B)/N.$
So $\tilde \iota$ is uniquely extended to a  contractive linear map
from $H\hspace{-0.03in}\bullet\hspace{-0.03in} B$ into $H_1.$ 
%with the same norm. 
It is a $B$-module map. 
Since $(H\otimes B)/N$ is dense in $H\hspace{-0.03in}\bullet\hspace{-0.03in}B,$
$$
\la \tilde \iota(x), \tilde \iota(y)\ra=\la x, y\ra\rforal x, y\in H\hspace{-0.03in}\bullet\hspace{-0.03in}B.
$$
To see this embedding is unique, let $\tilde \iota_1$ be another such embedding.
Then $(\tilde \iota-\tilde \iota_1)|_{H}=0.$
For any $\xi=\sum_{i=1}^n x_i\hspace{-0.03in}\bullet\hspace{-0.03in} a_i,$ 
where $x_i\in H$ and $a_j\in B,$
\beq
(\tilde \iota-\tilde \iota_1)(\xi)=\sum_{i=1}^n (\iota(x_i)-\iota(x_i))\hspace{-0.03in}\bullet\hspace{-0.03in}a_i=0.
\eneq
In other words, $\tilde \iota_1=\tilde \iota.$ 
\end{proof}

\begin{df}\label{Dmap2}
Keep the notations in \ref{D101}, \ref{DHsim} and \ref{Dmap1}.
Recall that $F(H)$ is  the algebra of all finite rank module maps.
\iffalse
i.e., linear span of all
bounded module maps 
of the form:
$\theta_{x, y}: H\to H,$ defined by
$$
\theta_{x, y}(\xi)=x\la y, \xi\ra
$$
for all $\xi\in H$ and $x, y\in H.$
\fi
Define $\Psi_0: F(H)\to F(H\hspace{-0.03in}\bullet\hspace{-0.03in} B)\subset B(H\hspace{-0.03in}\bullet\hspace{-0.03in} B)$ by
\beq
\Psi_0(\theta_{x, y})(\zeta)=\iota(x)\la \iota(y), \zeta\ra
\eneq
for all $\zeta\in H\hspace{-0.03in}\bullet\hspace{-0.03in} B,$ $x, y\in H.$  $\Psi$ is a $*$-preserving \hm\, from 
the $*$-algebra $F(H)$ into $F(H\hspace{-0.03in}\bullet\hspace{-0.03in} B).$ 
 Moreover,  $\Psi_0$ is an isometry on $F(H).$ In particular, 
$\|\Psi_0\|=1.$
Therefore it extends uniquely to a \CA\, \hm\, from $K(H)$ 
to $K(H\hspace{-0.03in}\bullet\hspace{-0.03in} B)$ (which preserves the norm).
It has to be an isometry as $F(H)$ is dense in $K(H).$

In the case that $B=A^{**},$ we may
define $\tilde \Psi_0: F(H)\to F(H^\sim)\subset B(H^\sim)$ by
\beq
\tilde\Psi_0(\theta_{x, y})(\zeta)=\iota(x)\la \iota(y), \zeta\ra
\eneq
for all $\zeta\in H^\sim,$ $x, y\in H.$  Then 
$\tilde\Psi_0$ is a $*$-preserving \hm\, from 
the $*$-algebra $F(H)$ into $F(H^\sim)$ 
% Moreover,  $\Psi_0$ is an isometry on $F(H).$ In particular, 
%$\|\Psi_0\|=1.$
%Therefore 
and it extends uniquely to a \CA\, \hm\, $\tilde\Psi_0$ from $K(H)$ 
to $K(H^\sim)$ which preserves the norm. 
Recall that $\iota(H^\sharp)\subset H^\sim.$ 
\end{df}

\begin{prop}\label{PTid}
Let $A\subset B$ be a pair of \CA s, where $B$ is unital
and $1_B=1_{\tilde A}.$ 
Let $T\in K(H).$ Then 
$\Psi_0(T)(x\hspace{-0.03in}\bullet\hspace{-0.03in} b)=T(x)\hspace{-0.03in}\bullet\hspace{-0.03in} b$
for all $x\in H$ and $b\in B.$ 
\end{prop}

\begin{proof}
From the definition, for any $S\in F(H),$
any $x\in H$ and any $b\in B,$
\beq
\Psi_0(S)(x\otimes b)=S(x)\otimes b\,\,\, ({\rm mod}\, N).
\eneq 
Fix $T\in K(H)$ and 
 let $\ep>0.$
 There exists $S\in F(H)$ such that
 \beq
 \|T-S\|<\ep/2(1+\|x\hspace{-0.03in}\bullet\hspace{-0.03in}  b\|+\|x\|\|b\|).
 \eneq
 Then  
 \beq
 \|\Psi_0(T)-\Psi_0(S)\|<\ep/4(1+\|x\otimes b\|+\|x\|\|b\|)\andeqn \|T(x)\hspace{-0.03in}\bullet\hspace{-0.03in} b-S(x)\hspace{-0.03in}\bullet\hspace{-0.03in} b\|\le \ep/2.
 \eneq
Hence 
\beq
\|\Psi_0(T)(x\hspace{-0.03in}\bullet\hspace{-0.03in}  b)-T(x)\hspace{-0.03in}\bullet\hspace{-0.03in} b\|<\ep.
\eneq
Since this holds for all $\ep>0,$ we conclude that
$$
\Psi_0(x\hspace{-0.03in}\bullet\hspace{-0.03in} b)=T(x)\hspace{-0.03in}\bullet\hspace{-0.03in} b.
$$
\end{proof}

\begin{lem}\label{Lelambda}
Let $A$  and $B$ be as in Proposition \ref{PTid}, 
and $H$ be a Hilbert $A$-module.
Suppose that $\{E_\lambda\}$ is an approximate identity for $K(H).$
Then
$
\{\Psi_0)(E_\lambda)\}$ forms an approximate identity for $K(H\hspace{-0.03in}\bullet\hspace{-0.03in} B).$
Moreover
\beq
\lim_{\lambda}\|\Psi_0(E_\lambda)(x)-x\|=0\rforal x\in H\hspace{-0.03in}\bullet\hspace{-0.03in} B.
\eneq
\iffalse
Furthermore, for any $T\in B(H^\sim),$
\beq
\lim_{\lambda}\|\la \Psi_0(E_\lambda) T \Psi_0(E_\lambda)(x), y\ra-\la T(x), y\ra\|=0\rforal x, y\in H\hspace{-0.03in}\bullet\hspace{-0.03in} A^{**}.
\eneq
\fi 
\end{lem}

\begin{proof}
By Lemma 3.1 of \cite{BL}, 
\beq\label{Lalambda-1}
\lim_{\lambda} \|E_\lambda (x)-x\|=0\rforal x\in H.
\eneq
Let $S=\sum_{i=1}^n \theta_{x_i, y_i},$ where 
$x_i, y_i\in 
%\hspace{-0.03in}\bullet\hspace{-0.03in} B,$
(H\otimes  B)/N,$ 
$1\le i\le n.$
Write $x_i=\sum_{j=1}^{k(i)} \xi_{j,i}\hspace{-0.03in}\bullet\hspace{-0.03in}  b_{j,i},$ where 
$\xi_{j,i}\in H$ and $b_{j,i}\in B,$ $j=1,2,..., k(i),$ $i=1,2,...,n.$ 
By Proposition \ref{PTid}, 
\beq
\Psi_0(E_\lambda)(\xi_{j,i}\hspace{-0.03in}\bullet\hspace{-0.03in} b_{j,i})=E_{\lambda}(\xi_{j,i})\hspace{-0.03in}\bullet\hspace{-0.03in} b_{j,i}.
%\,\,\, ({\rm mod}\, N).
\eneq
By \eqref{Lalambda-1},
\beq\label{Lalambda-2}
\lim_{\lambda} \|\Psi_0(E_\lambda)(\xi_{j,i}\hspace{-0.03in}\bullet\hspace{-0.03in}  b_{j,i})-(\xi_{j,i}\hspace{-0.03in}\bullet\hspace{-0.03in}  b_{j,i})\|=0
\eneq
for $j=1,2,..., k(i),$ $i=1,2,...,n.$ 
It follows that
\beq
\lim_{\lambda} \|\Psi_0(E_\lambda)(x_i)-x_i\|=0,\,\,\, i=1,2,...,n.
\eneq
For any $z\in H\hspace{-0.03in}\bullet\hspace{-0.03in} B,$
\beq
\Psi_0(E_\lambda)\theta_{x_i,y_i}(z)&=&(\Psi_0(E_\lambda)x_i)\la y_i, z\ra\\
&=&
E_\lambda(x_i)\la y_i, z\ra.
\eneq
It follows that (for $1\le i\le n$)
\beq
\lim_\lambda \|\Psi_0(E_\lambda)\theta_{x_i, y_i}-\theta_{x_i, y_i}\|=0.
\eneq
Hence 
\beq
\lim_{\lambda}\|\Psi_0(E_\lambda)S-S\|=0.
\eneq
The set of those  module maps with the form of $S$ is norm  dense in $K(H\hspace{-0.03in}\bullet\hspace{-0.03in} B).$  Therefore we conclude that
\beq
\lim_\lambda \|\Psi_0(E_\lambda)S-S\|=0\rforal S\in  K(H\hspace{-0.03in}\bullet\hspace{-0.03in} B).
\eneq
It follows that $\{\Psi_0(E_\lambda)\}$ forms an approximate identity for 
$K(H\hspace{-0.03in}\bullet\hspace{-0.03in} B).$

%The second part of the lemma follows easily from \eqref{Lalambda-2}.

\end{proof}

\iffalse
\begin{lem}\label{CTElambda}
Let $T\in B(H)$ and $\{E_\lambda\}$ be  an approximate identity for $K(H).$ 
Then 
\beq
\lim_{\lambda}TE_\lambda x
\eneq

\end{lem}
\fi

\begin{NN}\label{NN1}
{\rm Let $A$ be a \CA\, and $H$ be a Hilbert $A$-module.
Then $H^\sharp$ is a Banach $A$-module in general.
Recall that, for each $T\in B(H),$ one may define a bounded conjugate module map  $T^*: H\to H^\sharp$ 
as follows.
For $x, y\in H,$ define
\beq
T^*(x)(y)=\la x,T(y)\ra.
\eneq
So, for a fixed $x,$ $T^*(x)$ gives an element in $H^\sharp.$ 
Moreover, $T^*$ is a bounded conjugate module map from $H$ to $H^\sharp$ with $\|T^*\|=\|T\|.$
However, if we view $H$ as a submodule of $H^\sharp,$ then $T^*$ is a bounded module map. 
Note that, if $T\in L(H),$ then $T^*\in L(H)$ and $T^*(H)\subset H.$

If $A$ is a $W^*$-algebra, by Theorem 3.2 of \cite{Pa},  $H^\sharp$ becomes a Hilbert $A$ module
in a natural way.  
For $T\in B(H)$ and $f\in H^\sharp,$
define, for each $x\in H,$  
\beq
\tilde T(f)(x)=\la f, T^*(x)\ra, 
\eneq
where $T^*$ is defined above.
Thus $\tilde T(f)$ is a bounded linear module map from $H$ to $A$ with $\|\tilde T(f)\|\le \|T\|\|f\|.$
%Note that, if $f\in H,$ 
Hence we extend $T$ to a bounded (conjugate) module map from $H^\sharp$ to $H^\sharp.$
%
%If $A$ is a $W^*$-algebra, by Theorem 3.2 of \cite{Pa},  $H^\sharp$ becomes a Hilbert $A$ module
%in a natural way. 
As we view $H^\sharp$ as a Hilbert $A$-submodule in this case, $T$ is in fact a bounded module map
on $H^\sharp$ (we will take the conjugate as Hilbert space cases). 
By Corollary 3.7 of \cite{Pa}, such extension is unique. 

By Lemma 3.7 of \cite{Lininj}, one may ease the assumption that $A$ is a $W^*$-algebra to 
the assumption that $A$ is a monotone complete \CA.}

\end{NN}

\begin{prop}\label{PElsharp}
Let $A$  and $B$ be as in \ref{PTid},
$H$ be a Hilbert $A$-module and $\{E_\lambda\}$ an approximate identity 
for $K(H).$ 
%(1) Suppose that $B$ is a unital \CA\, containing $\tilde A$ with $1_{\tilde A}=1_B.$
Then 
\beq\label{232-n}
\lim_{\lambda}(\sup\{\|\tilde \Psi_0(E_\lambda)(f)(x)-f(x)\|: f\in (H\hspace{-0.03in}\bullet\hspace{-0.03in} B)^\sharp,\,\|f\|\le 1\})=0\rforal
% f\in (H\hspace{-0.03in}\bullet\hspace{-0.03in} B)^\sharp \andeqn 
x\in H\hspace{-0.03in}\bullet\hspace{-0.03in} B.
\eneq

Moreover,  suppose that $(H\hspace{-0.03in}\bullet\hspace{-0.03in} B)^\sharp $ 
extends $H\hspace{-0.03in}\bullet\hspace{-0.03in} B$ as Hilbert $B$-module,  then,
for any $T\in B((H\hspace{-0.03in}\bullet\hspace{-0.03in} B)^\sharp ),$
\beq
\lim_{\lambda}\|\la \tilde \Psi_0(E_\lambda) T \tilde\Psi_0(E_\lambda)(x), y\ra-\la T(x), y\ra\|=0\rforal x, y\in H\hspace{-0.03in}\bullet\hspace{-0.03in} B.
\eneq
\end{prop}

\begin{proof}
Fix $f\in H^\sharp.$ For any $x\in H\hspace{-0.03in}\bullet\hspace{-0.03in} B,$  by Lemma \ref{Lelambda}, 
%3.1 of \cite{BL}, 
\beq
\|\tilde \Psi_0(E_\lambda)(f)(x)-f(x)\|=\|f(E_\lambda(x))-f(x)\|\le \|f\|\|E_\lambda(x)-x\|\to 0.
\eneq
Hence \eqref{232-n} holds.

To see the ``Moreover" part of the lemma,  let $T\in  B((H\hspace{-0.03in}\bullet\hspace{-0.03in} B)^\sharp).$
Then, for any $x, y\in H\hspace{-0.03in}\bullet\hspace{-0.03in} B,$  
\beq\nonumber
&&\hspace{-0.8in}\|\la \tilde \Psi_0(E_\lambda) T \tilde \Psi_0(E_\lambda)(x), y\ra -\la T(x), y \ra \|\\
&&\le  \|\la T \tilde \Psi_0(E_\lambda) (x), \tilde \Psi_0(E_\lambda)(y)\ra -\la T(x), 
\tilde \Psi_0(E_\lambda)(y)\ra\|\\
&&\hspace{0.5in}+\|\la T(x), \tilde\Psi_0(E_\lambda)(y)\ra -\la T(x), y\ra\|\\
&&\le  \|y\|\|T\|\|\Psi_0(E_\lambda)(x)-x\|+\|T\| \|x\| \| \Psi_0(E_\lambda)(y)-y\|
\eneq
By applying  Lemma \ref{Lelambda} to the two terms of the last inequality above, 
we conclude that
\beq
\lim_{\lambda}\|\la \tilde \Psi_0(E_\lambda) T \tilde \Psi_0(E_\lambda)(x), y\ra -\la T(x), y \ra \|=0
\rforal 
%for all $
x, y\in H\hspace{-0.03in}\bullet\hspace{-0.03in} B.
\eneq
\end{proof}

\begin{df}\label{DB(H)}
Let $A$ be a \CA\, and $H$ be a Hilbert $A$-module.
Recall  (Theorem 1.5 of \cite{Linbd}) that we identify $B(H)$ with $LM(K(H)),$
the Banach algebra of left multipliers of $K(H)$ (in $K(H)^{**}$).

By \ref{Lelambda},  $\Psi_0$ maps $K(H)$ into $K(H\hspace{-0.02in}\bullet\hspace{-0.02in} B)$ 
which maps approximate identities  to approximate identities.
It follows that  we may extend  a \hm\, $\Psi_0: B(H)=LM(K(H))\to LM(K(H\hspace{-0.03in}\bullet\hspace{-0.03in}B))=B(H\hspace{-0.03in}\bullet\hspace{-0.03in} B)$ 
by
$$
\Psi_0(T)=\lim_\lambda \Psi_0(TE_\lambda),
$$
where the convergence is in the  left strict topology of $LM(K(H\hspace{-0.03in}\bullet\hspace{-0.03in} B)).$ 
Since ${\Psi_0}|_{K(H)}$ is an isometry, so is $\Psi_0.$ 

We are mostly interested in the case that $B=A^{**}.$ 
By Theorem 3.2 of \cite{Pa},  $(H\hspace{-0.03in}\bullet\hspace{-0.03in} A^{**}))^\sharp$  is a self-dual  Hilbert $A^{**}$-module.
Therefore, by \ref{NN1}, for each $T\in B(H),$  the extension ${\widetilde{\Psi_0(T)}}$
is unique. Hence $\Psi_0$  may be extended to a Banach algebra  isomorphism $\tilde \Psi_0$ from $B(H)$ into 
$B(H^\sim)$   such that
\beq\label{DB(H)-1}
\tilde \Psi_0(T)|_{H\hspace{-0.02in}\bullet\hspace{-0.02in} A^{**}}=\Psi_0(T)\rforal T\in B(H).
\eneq
We will visualize the map $\Psi_0$  a bit more.
\iffalse
For each $T\in B(H),$
$$
\lim_{\lambda} \|TE_\lambda-T\|=0.
$$
Let $\ep>0.$ There exists $\lambda_0$ such that, for any $\lambda\ge \lambda_0,$
\beq
\|TE_\lambda-T\|<\ep/2.
\eneq
Hence, for all $\lambda, \lambda'\ge \lambda_0,$  
\beq
\|TE_\lambda-TE_{\lambda'}|<\ep
\eneq
It follows that
\beq
\lim_{\lambda}\Psi_0(TE_\lambda)(x)
\eneq
exists.

For each $T\in B(H),$ define, for all $x\in H\hspace{-0.03in}\bullet\hspace{-0.03in} A^{**},$
\beq
\Psi_0(T)(x)=\lim_{\lambda} \Psi_0(TE_\lambda)(x).
\eneq
This defines a module map on $H\hspace{-0.03in}\bullet\hspace{-0.03in} A^{**}.$
Moreover $\|\Psi_0(T)\|\le \|T\|.$  So it is bounded module map.
If $T\in K(H)$
\fi
\end{df}

\begin{prop}\label{PPsi}
Let $A$  and $B$ be a  pair of \CA s as in Proposition \ref{PTid} and 
$H$ be a Hilbert $A$-module.
Then, for any $T\in B(H),$ 
\beq\label{PPsi-1}
\lim_{\lambda}\|\Psi_0(T)\Psi_0(E_\lambda)(x)-\Psi_0(T)(x)\|=0\rforal x\in H\hspace{-0.03in}\bullet\hspace{-0.03in} B.
\eneq
Moreover
%for any $x\in H$ and $b\in B,$
\beq
\Psi_0(T)(x\hspace{-0.03in}\bullet\hspace{-0.03in} b)=T(x)\hspace{-0.03in}\bullet\hspace{-0.03in}  b \rforal x\in H\andeqn b\in B.
\eneq 
Consequently, 
$\tilde \Psi_0({\rm id}_H)={\rm id}_{H^\sim}.$
\end{prop}

\begin{proof} 
The identity \eqref{PPsi-1} follows from immediately from Lemma \ref{Lelambda}.
\iffalse
We first note, by Lemma 3.1 of \cite{BL},  that  $\lim_{\lambda}\|E_\lambda(x)-x\|=0$ for all $x\in H.$
Let $x\in H$ and $b\in B.$  Then, for $T\in B(H),$ 
\beq\label{PPsi-3}
\lim_{\lambda} \|TE_\lambda (x)-T(x)\|=0\rforal x\in H.
\eneq
It follows that, for any $x\in H$ and $b\in B,$
\beq\label{PPsi-4}
\lim_{\lambda}\|(TE_\lambda(x)\otimes b)/N-(T(x)\otimes b)/N\|=0
\eneq
By Lemma \ref{Lelambda}, we also have 
\beq
\lim_{\lambda}\|\Psi_0(T)\Psi_0(E_\lambda)(z)-\Psi_0(T)(z)\|=0\rforal z\in H\hspace{-0.03in}\bullet\hspace{-0.03in} A^{**}.
\eneq
So the first part of the proposition follows.
\fi

Since 
\beq
\Psi_0(TE_\lambda)(x\hspace{-0.03in}\bullet\hspace{-0.03in} b)=TE_\lambda(x)\hspace{-0.03in}\bullet\hspace{-0.03in} b,
\eneq
by   \eqref{PPsi-1} and by  Lemma 3.1 of \cite{BL},
%Lemma \ref{Lelambda} again,
% \eqref{PPsi-4},
$$
\Psi_0(T)(x\hspace{-0.03in}\bullet\hspace{-0.03in} b)=T(x)\hspace{-0.03in}\bullet\hspace{-0.03in} b
$$
for all $x\in H$ and $b\in B.$

For the last part of the proposition,
we note that,  by considering the pair $A$ and $A^{**},$ and by the ``Moreover part" of the proposition, 
$\Psi_0({\rm id}_H)={\rm id}_{H\hspace{-0.02in}\bullet\hspace{-0.02in} A^{**}}.$ Therefore,  since the extension $\tilde \Psi_0({\rm id}_{H\hspace{-0.02in}\bullet\hspace{-0.02in} A^{**}})$ is unique 
(Corollary 3.7 of \cite{Pa}, see \ref{NN1} for convenience), we must have that $\tilde \Psi_0({\rm id}_H)={\rm id}_{H^\sim}.$ 
\end{proof}

The following is a slightly strengthened restatement of Proposition 2.3 of \cite{BL}.

\begin{prop}\label{PBL}
Let $A$ be a \CA\, and $H$ be a Hilbert $A$-module. 
Then there is a \hm\, $\Psi$ from $K(H)^{**}$ into  $B(H^\sim)$ 
such that $\Psi|_{B(H)}=\tilde \Psi_0.$  Moreover, if $T\in K(H)^{**}$ and $T_\lambda\in K(H)^{**}$ such that
$T_\lambda\to T$ in the weak*-topology, then 
\beq
\lim_{\lambda}f(\la \Psi(T_\lambda)(x), y\ra)=f(\la \Psi(T)(x), y\ra)\rforal x, y\in H^\sim \andeqn f\in A^*.
\eneq

\end{prop}

\begin{proof}
By  \ref{DHsim}, $B(H^\sim)=L(H^\sim)$ is a $W^*$-algebra (see Proposition 3.11 of \cite{Pa}).
Let $\pi: B(H^\sim)\to B(H_\pi)$ be a faithful normal representation
such that $\pi(B(H^\sim))$ is weakly closed in $B(H_\pi).$
 Then, by, for example, 
 %Theorem 1.8.2 of \cite{Lnbk} (see also 
 Theorem  3.7.7 of \cite{Pedbook} and  Cor. 46.5 of \cite{Conway}, there is a normal \hm\, 
 $\Phi: K(H)^{**}\to B(H_\pi)$ such that $\Phi|_{K(H)}=\pi\circ {\tilde \Psi_0}|_{K(H)}$ 
 %(as we view $K(H)\subset B(H^\sim)$)
 and $\pi\circ \tilde \Psi_0(K(H))$ is weakly dense in $\Phi(K(H)^{**}).$ Since $\pi(B(H^\sim))$ is a von Neumann algebra,
 $\Phi(K(H)^{**})\subset \pi(B(H^\sim)).$ Since $\pi$ is injective,  we may
 define $\Psi=\pi^{-1}\circ \Phi.$   Recall that $\pi^{-1}$ is an isomorphism 
 between $W^*$-algebras $\pi(B(H^\sim))$ and $B(H^\sim).$ 
 It follows that $\Psi$ is weak *-continuous. 
 Then, $\Psi|_{K(H)}=\pi^{-1}\circ {\pi\circ \tilde \Psi_0}|_{K(H)}={\tilde \Psi_0}|_{K(H)}.$

Let $V=B(H^\sim)_*$ be the pre-dual (as Banach spaces). Then  $\Psi$ induces a map
$\Psi^*: V\to K(H)^*,$ the pre-dual of $K(H)^{**},$  by
$L(\Psi^*(v))=\Psi(L)(v)$ for all $L\in (K(H)^*)^*$ and $v\in V.$
Thus if $T_\lambda\in K(H)^{**}$ such that $T_\lambda\to T$ in the weak *-topology in $K(H)^{**},$ then
$\Psi(T_\lambda)(v)=T_\lambda(\Psi^*(v))$ converges to $T(\Psi^*(v))=\Psi(T)(v)$
for all $v\in V.$ In other words, 
$\Psi(T_\lambda) \to \Psi(T)$ in the weak *-topology   in $V^*=B(H^\sim).$
By \ref{DHsim} (see  Remark 3.9  and proof of Theorem 3.10 of \cite{Pa}), this implies, in particular,  
for any $f\in A^*,$ $x, y\in H^\sim,$
$f(\la \Psi(T_\lambda((x), y\ra)\to f(\la \Psi(T)(x), y\ra).$ 

By  Theorem 1.5  of \cite{Linbd}, 
%Theorem 1 of \cite{K}, 
%
$B(H)=LM(K(H)).$ Let $\{E_\lambda\}$ be an approximate identity for $K(H).$
Then $TE_\lambda\in K(H)$ for all $T\in B(H).$ 
It follows from Proposition \ref{PPsi} that, for $T\in B(H),$  
\beq
\lim_{\lambda}\|\Psi(TE_\lambda)(f)(x)-\Psi(T)(f)(x)\|
=\lim_{\lambda}\|\Psi(T)\Psi(E_\lambda)(f)(x)-\Psi(T)(f)(x)\|=0
\eneq
for all $x\in H\hspace{-0.03in}\bullet\hspace{-0.03in} A^{**}$ and $f\in (H\hspace{-0.03in}\bullet\hspace{-0.03in}  A^{**})^\sharp.$ 
On the other hand,  by Proposition \ref{Lelambda}, 
\beq
\lim_{\lambda} \|\Psi_0(TE_\lambda)(x)-\Psi_0(T)(x)\|=0\rforal x\in H\hspace{-0.03in}\bullet\hspace{-0.03in} A^{**}.
\eneq
%Recall that $\Psi_0(TE_\lambda)\to \Psi_0(T)$ in the strictly topology 
%in $B(H\hspace{-0.03in}\bullet\hspace{-0.03in} A^{**}).$ 
%and by what has just been proved
%\beq
%\lim_{\lambda} |f(\la \Psi(TE_\lambda)(x), y\ra)-f(\Psi(T)(x), y\ra)|=0
%\|\tilde \Psi(TE_\lambda)(x)-\tilde \Psi(T)(x)\|
%\rforal x, y\in H^\sim \andeqn f\in A^*.
%\eneq
However, we have shown $\Psi(TE_\lambda)(y)=\tilde \Psi_0(TE_\lambda)(y)=\Psi_0(TE_\lambda)(y)$
for all $y\in H\hspace{-0.03in}\bullet\hspace{-0.03in} A^{**}$ (see also \eqref{DB(H)}).
Therefore, combining these three facts, for $x, y\in H\hspace{-0.03in}\bullet\hspace{-0.03in} A^{**},$  we obtain 
\beq
\la \Psi(T)(x), y\ra =\la \Psi_0(T)(x), y\ra.
\eneq
It follows that $\Psi(T)|_{H\hspace{-0.01in}\bullet\hspace{-0.01in} A^{**}}=\Psi_0(T).$
%Therefore, for $x\in H\hspace{-0.03in}\bullet\hspace{-0.03in} A^{**},$
%\beq
%\la \Psi(T)(y), x\ra=\la \Psi_0(T)(y), x\ra \rforal x, y\in H\hspace{-0.03in}\bullet\hspace{-0.03in} A^{**}.
%\eneq
%It follows that $\Psi(T)(y)=\tilde\Psi_0(T)(y)$ for $y\in H\hspace{-0.03in}\bullet\hspace{-0.03in} A^{**}.$
Since the extension of $\Psi_0(T)$ to a bounded module map on $(H\hspace{-0.03in}\bullet\hspace{-0.03in} A^{**})^\sharp$
is unique (see the end of \ref{NN1}, also   Lemma 3.5 of \cite{Lininj}), we have  $\Psi(T)=\tilde \Psi_0(T)$ for all $T\in B(H).$ Hence 
$$
\Psi|_{B(H)}=\tilde \Psi_0.
$$
\iffalse
To see $\Psi$ is injective, let $T\in B(H)^{**}$ be a non-zero element.
We obtain a net $T_\lambda\in K(H)$ 
such that   $T_\lambda\to T$ in the weak*-topology.

There must be $f\in H^\sim$ such that $T(f)\not=0.$
It follows that there is $x\in H\hspace{-0.03in}\bullet\hspace{-0.03in} A^{**}$ such that
$T(f)(x)\not=0.$ 
\fi
%%% By Theorem 1 of \cite{K}, $B(H)=LM(K(H)).$
%Hence, if $T\in B(H),$ then $Te_\lambda \to T$ in the  right strictly topology. One then verifies that
%$\Psi|_{B(H)}={\rm id}|_{B(H)}$ where we identify $B(H)$ with a subset of $B(H^\sim).$
 %
\end{proof}

\begin{df}\label{DFW}
Let $M$ be a $W^*$-algebra and $H$ be a Hilbert $M$-module.
Then, by 3.2 of \cite{Pa}, $H^\sharp$ is a  self-dual Hilbert $M$-module.
Let $F_0: F(H)\to F(H^\sharp)$ be the \hm\, 
defined by
$$
F_0(\theta_{x, y})(z)=x\la y, z\ra \rforal z\in H^\sharp\andeqn x, y\in H.
$$
Clearly $F_0$ is an isometry. It extends 
uniquely a \hm\, $F_0: K(H)\to K(H^\sharp).$ 
We  further extend $F: \widetilde{K(H)}\to \widetilde{K(H^\sharp)}$
by $F({\rm id}_H)={\rm id}_{H^\sharp}.$ 
\end{df}

\begin{prop}\label{s2-adp}
Let $M$ be a $W^*$-algebra and $H$ be a Hilbert 
$M$-module. Then there exists 
a unital normal \hm\, $F: K(H)^{**}\to B(H^\sharp)$ 
such that
$F|_{K(H)}=F_0$ and, if $T_\lambda\to T$ in weak*-topology of $W,$ then
\beq
\lim_{\lambda} f(\la F(T_\lambda)(x), y\ra)=f(\la F(T)(x), y\ra)
\eneq
for all $x, y\in H^\sharp$ and $f\in M_*,$ the pre-dual of $M.$
Moreover, $F(T)=\tilde T$ for all $T\in B(H)$ as defined in the later part of \ref{NN1}.
\end{prop}
\begin{proof}
Recall that $B(H^\sharp)$ is a $W^*$-algebra. 
We may assume that $B(H^\sharp)$ acting on a Hilbert space $X$  as a von Neumann algebra 
with $1_{B(H^\sharp)}={\rm id}_X.$ 
Then, by Theorem 1.8.2 of \cite{Lnbk} (see also Theorem  3.7.7 of \cite{Pedbook}), there is a unital normal \hm\, 
$F: K(H)^{**}\to \overline{F_0(K(H))}^{SOT}\subset B(H^\sharp)$
such that $F|_{K(H)}=F_0.$ So $F$ is weak*-continuous (see, for example, Cor. 46.5 in \cite{Conway}). 

Suppose that  $T_\lambda\to T$ in the weak*-topology of $K(H)^{**},$
then  $F(T_\lambda)\to F(T)$ in the weak*-topology of $B(H^\sharp).$ 
Therefore (see the later part of Definition \ref{DHsim}, also, Remark 3.9 and  the proof of Proposition 3.9 of \cite{Pa}), 
\beq
f(\la F(T_\lambda)(x), y\ra)\to f(\la F(T)(x), y\ra)
\rforal x, y\in H^{\sharp} \andeqn f\in M_*.
\eneq 
%
%The ``Moreover" part follows 
Let $\{E_\lambda\}$ be an approximate identity for $K(H).$ 
Then, for any $T\in B(H),$ by Lemma 3.1 of \cite{BL}, 
\beq
\lim_\lambda \|F(T)F(E_\lambda)(x)-F(T)(x)\|=\lim_\lambda \|F(T)E_\lambda(x)-F(T)(x)\|=0\rforal x\in H.
\eneq
On the other hand, since $F(T)F(E_\lambda)|_H=F(TE_\lambda)|_H=TE_\lambda,$  and  
\beq
\lim_\lambda \|TE_\lambda(x)-T(x)\|=0,
\eneq
we conclude that
\beq
T(x)=F(T)(x)\rforal x\in H.
\eneq
Since the extension of $T$ to $H^\sharp$ is unique (by Proposition 3.6 of \cite{Pa},
see also Lemma 3.5 of \cite{Lininj}),  $\tilde T=F(T).$ 
\end{proof}

\section{Isomorphism of $B(H^\sim)$  and $K(H)^{**}$ }

Let $A$ be a monotone \CA\, and $H$ be a Hilbert $A$-module.
Then, by Lemma 3.7 of \cite{Lininj}, $H^\sharp$  becomes 
a  self-dual  Hilbert 
$A$-module such that $\la \tau, x\ra=\tau(x)$
for all $x\in H$ and $\tau\in H^\sharp.$   Note that, if $E$ is self-dual,
we map $E^\sharp$ onto $E$ conjugately 
just as in the case of Hilbert spaces.

We will apply the following lemma several times.

\begin{prop}\label{ext}
Let $A$ be a monotone \CA, $H_1\subset H_2$ be 
Hilbert $A$-modules such that $H_2$ is self-dual.
Then $H_1^\sharp$ is an orthogonal summand of
$H_2^\sharp$ and the embedding $H_1^\sharp\to H_2^\sharp$ extends the embedding 
$H_1\subset H_2.$
\end{prop}

\begin{proof}

Define $P_0: H_2\to H_1^\sharp$ by
\beq\label{ext-1}
P_0(y)(x)=\la y, x\ra \rforal y\in H_2\andeqn x\in H_1.
\eneq
It is a bounded module  map (by viewing $H_1^\sharp$ as a Hilbert module 
instead of the dual to avoid the conjugation) with $\|P_0\|=1.$ 
Note that ${P_0}|_{H_1}=\id_{H_1}.$

Let $\tau\in H_1^\sharp.$ Since $A$ is monotone complete, by
Theorem 3.8 of \cite{Lininj}, there is $\tilde \tau\in H_2^\sharp$
such that $\tilde \tau|_{H_1}=\tau$ and $\|\tilde \tau\|=\|\tau\|.$ 
This implies that $P_0$ is surjective.

Define $j: H_1^\sharp\to H_2^\sharp=H_2$ by
\beq\label{ext-2}
j(x)(y)=\la x, P_0(y)\ra \rforal x\in H_1^\sharp\andeqn y\in H_2.
\eneq
Then $j$ extends the embedding $H_1\subset H_2.$
Now, for $x\in H_1^\sharp$ and $y\in   H_2,$ by \eqref{ext-1} and \eqref{ext-2},
\beq
P_0\circ j(x)(y)&=&P_0(j(x))(y)=\la j(x), y\ra =\la x, P_0(y)\ra\\
&=& \la P_0(y), x\ra^*=(P_0(y)(x))^*\\
&=&\la y, x\ra^*
= \la x, y\ra.
\eneq
It follows that $P_0\circ j=\id|_{H_1^\sharp}.$ 
Thus $j: H_1^\sharp\to H_2$ is an emebeding.
With identification of $H_1^\sharp$ and $j(H_1^\sharp),$ 
${P_0}|_{H_1^\sharp}=\id|_{H_1^\sharp}.$ It follows that 
$P_0$ is a  projection and $H_1^\sharp$ is an orthogonal summand of $H_2.$
\end{proof}

Applying Lemma \ref{ext} and Proposition \ref{Psmall}, we obtain the  following characterization of $H^\sim.$

\begin{prop}\label{Psmall2}
Let $A$ be a \CA\, and $H$ be a Hilbert $A$-module.
Then $H^\sim$ is the smallest self-dual Hilbert $A^{**}$-module
containing $H$ as a Hilbert $A$- submodule.
\end{prop}

\begin{proof}
Let $H_1$ be a self-dual Hilbert $A^{**}$-module containing $H$ as a Hilbert $A$-submodule.
Then, by Proposition \ref{Psmall}, 
\beq
H\subset H\hspace{-0.02in}\bullet \hspace{-0.02in} A^{**}\subset H_1.
\eneq
Applying Lemma \ref{ext},  since $H_1$ is self-dual, 
\beq
H^\sim=(H\hspace{-0.02in}\bullet \hspace{-0.02in} A^{**})^\sharp\subset H_1^\sharp=H_1.
\eneq
The proposition follows.
\end{proof}

\begin{NN}
{\rm In the next proposition, let $A$ be a \CA,  $H_1\subset H$ be Hilbert $A$-modules.
Then, by Proposition \ref{Psmall}, $H_1\hspace{-0.03in}\bullet\hspace{-0.03in} A^{**}\subset H\hspace{-0.03in}\bullet\hspace{-0.03in} A^{**}.$ 
Since $A^{**}$ is monotone complete, and 
$(H\hspace{-0.03in}\bullet\hspace{-0.03in} A^{**})^\sharp=H^\sim$ and $(H_1\hspace{-0.03in}\bullet\hspace{-0.03in} A^{**})^\sharp=H_1^\sim,$   by Lemma \ref{ext}, we may write 
$H^\sim =H_1^\sim \oplus (H_1^\sim)^\perp.$ Denote 
by $P: H^\sim \to H_1^\sim $ the projection.  Note that $P\in L(H^\sim).$
By Lemma 3.2 of \cite{Lininj}, $K(H_1)$ is a hereditary \SCA\, of $K(H).$
\iffalse
As before, we will identify with $L(H)$ with $M(K(H))$ and use 
the same notation $P$ 
for the module map $P$ defined above as well as an element $P$ in 
$M(K(H))\subset K(H)^{**}.$ 
\fi
 Let  $\Psi_H: K(H)^{**}\to B(H^\sim)$ 
and  $\Psi_1: K(H_1)^{**}
\to B(H_1^\sim)$ 
%and $\Psi_2: K(H_2)^{**}\to H_2^\sim$
 be the  
\hm s given by Proposition \ref{PBL}, respectively.}
\end{NN}

\begin{prop}\label{Ladd}
%Let $A$ be a \CA, $H_1,$ $H_2$ be Hilbert $A$-modules.
Use the notation above,  we have
that
\iffalse
\beq
(H_1\oplus H_2)^\sim =H_1^\sim \oplus H_2^\sim.
\eneq
Moreover, 
\fi
$$
\Psi_1={\Psi_H}|_{K(H_1)^{**}}=P\Psi_HP |_{K(H_1)^{**}},
$$ in particular, $\Psi_1(T)=\Psi_H(T)|_{H_1^\sim}=P\Psi_H(T)P|_{H_1^\sim }$
for  $T\in K(H_1)^{**}.$ Moreover,
$$P\Psi_H(L)P|_{K(H_1)^{**}}\subset 
\Psi_1(K(H_1)^{**})\rforal L\in K(H)^{**}.$$ 
%where $P: H\to H_1$ is the projection in $L(H).$ 
Furthermore, 
$
\Psi(Q)=P,$ where $Q$ is the open projection in $K(H)^{**}$ corresponding to the hereditary \SCA\, $K(H_1).$ 
\end{prop}

\begin{proof}
\iffalse
It is self evident from the definition that
$$
H\hspace{-0.03in}\bullet\hspace{-0.03in} A^{**}=H_1\hspace{-0.03in}\bullet\hspace{-0.03in} A^{**}\oplus H_2\hspace{-0.03in}\bullet\hspace{-0.03in} A^{**}.
%(H\otimes A^{**}/N)^-=(H_1\otimes A^{**}/N)^-\oplus (H_2\oplus A^{**}/N)^-.
$$
It then follows that 
$$
H^\sim=H_1^\sim\oplus H_2^\sim.
$$
\fi

%By Proposition \ref{PBL}, we may write $\Psi(P)=P.$
%We will use the identity $K(H_1)^{**}=PK(H)^{**}P.$
 
%$$
%\Psi(K(H_1)^{**})|_{H_1^\sim}=\Psi(PK(H)^{**}P)|_{H^\sim}=P\Psi(K(H)^{**})P|_{H^\sim}.
%$$

Denote by $\Psi_{K(H),0}$ the injective \hm\,  from $K(H)$ into 
$K(H\hspace{-0.03in}\bullet\hspace{-0.03in} A^{**}),$ 
$\Psi_{K(H_1),0}$ the injective \hm\, 
from $K(H_1)$ into $K(H_1\hspace{-0.03in}\bullet\hspace{-0.03in} A^{**})$ described in \ref{Dmap2}, respectively. 

Fix $S\in K(H_1).$
For each $x\in H_1$ and $b\in A^{**},$
by \ref{PTid}, 
\beq
&&\Psi_{K(H),0}(S)(x\hspace{-0.03in}\bullet\hspace{-0.03in} b)=S(x)\hspace{-0.03in}\bullet\hspace{-0.03in} b\andeqn\\
&&P\Psi_{K(H),0}(S)P(x\hspace{-0.03in}\bullet\hspace{-0.03in} b)=P(S(x\hspace{-0.03in}\bullet\hspace{-0.03in} b))=S(x)\hspace{-0.03in}\bullet\hspace{-0.03in} b=\Psi_{K(H_1),0}(S)(x\hspace{-0.03in}\bullet\hspace{-0.03in} b).
\eneq
It follows that 
$$
\Psi_{K(H),0}(S)|_{H_1\hspace{-0.01in}\bullet\hspace{-0.01in} A^{**}}=P\Psi_{K(H),0}(S)P|_{H_1\hspace{-0.01in}\bullet\hspace{-0.01in} A^{**}}=\Psi_{K(H_1),0}(S).
$$
Since the extensions of $\Psi_{K(H),0}(S)|_{H_1\hspace{-0.01in}\bullet\hspace{-0.01in} A^{**}}$ and $\Psi_{K(H_1),0}(S)$ to bounded module maps on $H_1^\sim$ 
are unique, and,  $\Psi(S)|_{H_1^\sim}$  and $\Psi(S)$ are corresponding extensions, by Corollary 3.7 of \cite{Pa},
we conclude that 
$\Psi(S)|_{H_1^\sim}=P\Psi_H(S)P|_{H_1^\sim}=\Psi_1(S).$ 

Let $T\in K(H_1)^{**}$ and $\{T_\lambda\}\subset K(H_1)$ be a net 
such that $T_\lambda\to T$ in the weak*-topology. 
%It follows that $f(PT_\lambda P)\to f(PTP)$ for all $f\in A^*.$
%Then $PT_\lambda P=T_\lambda.$
%Denote by $S_\lambda$ the module map $PT_\lambda P|_{H_1}.$
%Then $S_\lambda\in K(H_1).$ 
By Proposition \ref{PBL}, 
for any $g\in A^*,$
\beq\label{Ladd-08}
&&\lim_{\lambda}|g(\la \Psi_H(T_\lambda)(x), y\ra)-g(\la \Psi_H(T)(x), y\ra)|=0
\andeqn\\\label{Ladd-09}
&&\lim_{\lambda}|g(\la \Psi_1(T_\lambda)(x), y\ra)-g(\la \Psi_1(T)(x), y\ra)|=0
\eneq
for all $x, y\in H_1^\sim.$
Since  we have shown that $\Psi_H(T_\lambda)|_{H_1^\sim}=P\Psi_H(T_\lambda)P|_{H_1^\sim}=\Psi_1(T_\lambda),$ 
we conclude that
\beq\label{Ladd-10}
\Psi_H(T)|_{H_1^\sim}=P\Psi_H(T)P|_{H_1^\sim}=\Psi_1(T).
\eneq
Hence 
\beq
\Psi_1=P\Psi_HP |_{K(H_1)^{**}}=\Psi_H|_{K(H_1)^{**}}.
\eneq
%Since $PK(H)^{**}P=K(H_1)^{**},$ 

Let $\{q_\lambda\}$ be an approximate identity for 
$K(H_1).$
Then $q_\lambda\nearrow {\rm id}_{H_1}\in K(H_1)^{**}.$
It follows from Proposition \ref{PBL} that
\beq
\lim_\lambda f(\la \Psi_1(q_\lambda(y)), z\ra)=f(\la y, z\ra)\rforal y, z\in H_1^\sim\andeqn f\in A^{*}.
\eneq

On the other hand, we also have that  $q_\lambda\nearrow Q$ in $K(H)^{**}.$
By Lemma  \ref{ext}, $H^\sim=H_1^\sim \oplus (H_1^\sim)^\perp.$
Note that $q_\lambda(x)\in H_1$ for all $x\in H.$
Then, for $x\in H,$ $b\in A^{**},$ and $g\in (H_1^\sim)^\perp,$   by Proposition \ref{PTid}, 
\beq
\la \Psi_H(q_\lambda) (x\hspace{-0.03in}\bullet\hspace{-0.03in} b), g\ra= g(q_\lambda(x\hspace{-0.03in}\bullet\hspace{-0.03in} b))^* =g(q_\lambda(x)\hspace{-0.03in}\bullet\hspace{-0.03in} b)^*=0.
\eneq
It follows that, for any $y\in H\hspace{-0.03in}\bullet\hspace{-0.03in} A^{**}$  and $g\in (H_1^\sim)^\perp,$
\beq
\la \Psi_H(q_\lambda)(y), g\ra=0.
\eneq
Hence, for $g\in (H_1^\sim)^\perp,$ 
\beq
\la y, \Psi_H(q_\lambda)(g)\ra=0\rforal y\in H\hspace{-0.03in}\bullet\hspace{-0.03in} A^{**}.
\eneq
It follows that $\Psi_H(q_\lambda)(g)=0$ and 
\beq
\la \Psi_H(q_\lambda)(z), g\ra=\la z, \Psi_H(q_\lambda)(g)\ra=0\rforal z\in H^\sim.
\eneq
In other words, $\Psi_H(q_\lambda)(z)\in H_1^\sim$ for all $z\in H$ and $\lambda.$
Therefore
\beq
P\Psi_H(q_\lambda)=\Psi_H(q_\lambda)=\Psi_H(q_\lambda) P.
\eneq
Note that $Pz\in H_1^\sim$ for any $z\in H^\sim.$
Thus, by \eqref{Ladd-10} and \eqref{Ladd-09},
\beq
 \lim_{\lambda} f(\la \Psi(q_\lambda)(y), z\ra)&=&
 \lim_{\lambda}f(\la \Psi(q_\lambda)(P(y)), P(z)\ra\\
 &=&\lim_\lambda f(\la \Psi_1(q_\lambda)(P(y)), P(z)\ra\\
 %=f(\la \Psi_1(q)(P(y)), P(z)\ra)\\
&=&  f(\la P(y), P(z)\ra)=f(\la P(y), z\ra).
\eneq
By \eqref{Ladd-08} and \eqref{Ladd-10}, $\lim_{\lambda} f(\la \Psi(q_\lambda)(y), z\ra)=f(\la \Psi(Q)(y), z\ra).$
Therefore 
\beq
\Psi(Q)=P.
\eneq
This proves the ``Furthermore" part.   In what follows we will identify $Q$ with $P$ as well as 
$\Psi(Q)$ and $\Psi(P).$

Now let $L\in K(H)^{**}$ and $\{L_\lambda\}\subset K(H)$
be a net such that $L_\lambda\to L$ in weak *-topology.
%It follows that $PL_\lambda P|_{H_1}\in K(H_1)$ for each $\lambda.$
%We also have that $f(PL_\lambda P)\to f(PLP)$ for all $f\in A^*.$
By Proposition \ref{PBL}, 
for any $g\in A^*,$ $x, y\in H_1^\sim,$ 
\beq
%&& \la \Psi_H(PT_\lambda P)(x), y\ra=\la \Psi_H(T_\lambda)(x), y\ra,\\
%&&\la P\Psi_H(T)P(x), y\ra=\la \Psi_H(T)(x), y\ra,\\
&&\lim_{\lambda}|g(\la \Psi_H(T_\lambda )(x), y\ra)-g(\la \Psi_H(T)(x), y\ra)|=0
\andeqn\\
&&\lim_{\lambda}|g(\la \Psi_1(T_\lambda )(x), y\ra)-g(\la \Psi_1(T)(x), y\ra)|=0
\eneq
(note that $\Psi_1(T_\lambda)=P\Psi_1(T_\lambda)P$).
We also have, for any $x, y\in H_1^\sim,$  
\beq
&& \la \Psi_H(PT_\lambda P)(x), y\ra=\la \Psi_H(T_\lambda)(x), y\ra,\\
&&\la P\Psi_H(T)P(x), y\ra=\la \Psi_H(T)(x), y\ra,\\
\eneq
Since $PT_\lambda P\in K(H_1)^{**},$ 
by the first part of the lemma, 
$\Psi_H(PT_\lambda P)(x)=\Psi_1(PT_\lambda P)(x)$
for $x\in H_1^\sim.$ 
It follows that $P\Psi_H(T)P(x)=\Psi_1(PTP)(x)$
for all $x\in H_1^\sim.$ Then $P\Psi_H(T)P=\Psi_1(PTP)\in \Psi_1(K(H_1)^{**}).$
%
 \iffalse
For $x\in H_1\oplus H_2=H,$ 
we may write $x=x_1\oplus x_2,$ where $x_i\in H_i,$ $i=1,2.$
Fix such $x$ and let $b\in A^{**}.$ 
We have, for any $T\in K(H_1),$   by ?, 
\beq
\Psi(T)(x\otimes  b)=\Psi_0(T)(x)\otimes b= T(x)\otimes b
=PTP(x)\otimes b.
\eneq
Hence $\Psi_0(T)|_{H_1\hspace{-0.03in}\bullet\hspace{-0.03in} A^{**}}=P\Psi_0(T)$
\fi
\end{proof}

\begin{NN}\label{DAsharp}
{\rm Let $A$ be a \CA\, and let, for $n\in \N,$ 
$H_n=A^{(n)}=\{(a_1, a_2,...,a_n): a_j\in A, 1\le j\le n\},$
the direct sum of  $n$-copies of $A,$ 
where 
$\la a, b\ra=\sum_{j=1}^n a_j^*b_j,$
if $a=(a_1, a_2,...,a_n)$ and $b=(b_1, b_2,...,b_n).$ 
Let 
\beq
H_A=\{\{a_n\}: a_n\in A\andeqn \sum_{i=1}^n a_k^*a_k\,\,\,{\rm  converges\,\, in\,\, norm}\}
\eneq
be the standard  countably generated Hilbert (right) $A$-module.
Note that 
\beq
\la \{a_n\}, \{b_n\}\ra =\sum_{n=1}^\infty a_n^*b_n.
\eneq
We note that $H_A$ is the closure of $\cup_n A^{(n)}.$
We  may also view $H_n=A^{(n)}$ as an orthogonal summand of $H_A.$
Then 
\beq
H_A^\sharp=\{\{a_n\}: \{\|\sum_{k=1}^n a_k^* a_k\|\} \,\, {\rm bounded}\}.
\eneq
If $g=\{a_n\}\in H^\sharp,$ then 
\beq
g(x)=\sum_{n=1}^\infty a_n^* b_n\rforal x=\{b_n\}\in H_A,
\eneq
where the sum converges in norm.
Moreover  $\|g\|=\lim_{n\to\infty}\|\sum_{k=1}^n b_k^*b_k\|.$
}

{\rm If $A$ is a $W^*$-algebra, as mentioned earlier, $H_A^\sharp$  becomes a Hilbert $A$-module in a natural way (see Theorem 3.2 of \cite{Pa}).  In fact, 
we may define 
\beq\label{NN-converge-0}
\la x, y\ra =\sum_{n=1}^\infty a_n^*b_n\rforal  x=\{a_n\},\,\, y=\{b_n\}\in H_A^\sharp.
\eneq
To see the right side converges 
in the weak*-topology, we first let $f\in A^*.$  
%be a positive linear functional.
%weakly, let  $A$ acting on a Hilbert space $H_\pi$
%as a von Neumann -algebra. Denote by $f(a)=\la \pi(a)(v), v\ra$
%for all $a\in A,$ where $v\in H_\pi$ is a vector.  Hence $f$ is a positive linear functional.
%let $f\in A_*$ be a  normal linear functional.
Note that, if $\{a_n\}\in H_A^{\sharp},$ 
\beq
|\sum_{k=1}^N f(a_k^*a_k)|=|f(\sum_{k=1}^N a_k^*a_k)|\le \|f\|\|\sum_{k=1}^N a_k^*a_k\|
\eneq
for any integer $N.$
Hence $\{\sum_{k=1}^n f(a_k^*a_k)\}$ is bounded, increasing and 
converges for any positive linear functional $f.$  Hence, for any $m>n,$ 
\beq\label{NN-converge}
\sum_{k=n}^m f(a_k^*a_k)\to 0\,\,\, {\rm as}\,\, n\to\infty\rforal f\in A^*.
\eneq
}
{\rm 
For any positive linear functional $f$ and for any $ m>n$ in $\N,$
\beq
|f(\sum_{k=n}^m a_k^*b_k)|&=&
|\sum_{k=n}^m f(a_k^*b_k)|\le \sum_{k=n}^m |f(a_k^*b_k)|\\
&\le &\sum_{k=n}^m |f(a_k^*a_k)|^{1/2} |f(b_k^*b_k)|^{1/2}\\
&\le & ((\sum_{k=n}^m |f(a_k^*a_k)|)(\sum_{k=n}^m |f(b_k^*b_k)|))^{1/2}\\
&\le & \|f\|^{1/2} \|\{b_k\}\||(\sum_{k=n}^m |f(a_k^*a_k)|)^{1/2}\to 0\,\,{\rm as}\,\, n\to\infty.
\eneq
It follows that 
$
f(\sum_{k=1}^na_k^*b_k)
$
converges for all $f\in A^*$ as $n\to\infty.$
Let us write the limit as $f(\sum_{k=1}^\infty a_k^*b_k).$ 
Then, by the above inequalities (with $n=1$), we also have
\beq
|f(\sum_{k=1}^\infty a_k^*b_k)|\le \|f\|M_bM_a,
\eneq
where 
\beq
M_a=\sup\{\|\sum_{k=1}^n a_k^*a_k\|\}^{1/2}\andeqn 
M_b =\sup\{\|\sum_{k=1}^n b_k^*b_k\|\}^{1/2}.
\eneq
Thus $\sum_{k=1}^\infty a_k^*b_k$ defines 
a bounded linear functional on $A^*.$ Its restriction on $A_*$ gives an element
in $A$ (recall that $A$ is assumed to be a $W^*$-algebra). 
%Thus 
%By the Jordan decomposition
% and, for example, Proposition 3.6.2 of \cite{Pedbook}, 
%the above also holds for any self-adjoint 
%linear functionals $f.$ 
%Thus, for any $f\in A^*,$  for any $n\in \N,$ 
%\beq
%|f(\sum_{i=1}^na_i^*b_i)|\le \|f\|\|\{b_k\}\|(\sum_{k=n}^m |f(a_k^*a_k)|)
%\eneq
This shows the infinite series in the right side of \eqref{NN-converge-0} converges 
in the weak*-topology.  It is then  standard to verify
that \eqref{NN-converge-0} defines an  inner product which extends the inner product 
on $H_A.$}

{\rm Let $A$  act on Hilbert space $X$  (as a $W^*$-algebra).
%as a von-Neumann algebra. 
  Consider $l^2(X)$ the Hilbert space direct sum 
of countably many copies of $X.$ Suppose that $b=\{b_n\}\in H_A^\sharp.$
Then the infinite matrix ${\bar b}=(b_{i,j})$ with $b_{i,1}=b_i,$ $i\in \N,$ and 
$b_{i, j}=0,$ if $j\ge 2,$ defines a bounded linear operator on $l^2(X),$
by ${\bar b}(v)=(b_1(v_1), b_2(v_1),...,b_n(v_1),...),$
where $v=(v_1, v_2,...,v_n,...)\in l^2(X).$  Moreover
\beq\label{bounded}
\|{\bar b}\|^2=\|{\bar b}^*{\bar b}\|=\|\sum_{i=1}^\infty b_i^*b_i\|=\sup\{\|\sum_{i=1}^n b_i^*b_i\|: n\in \N\}
\eneq
(some of these details  in this subsection may be found in \cite{Lin91module}). 
}

\end{NN}

\begin{prop}\label{HAB}
Let $C$ be a unital \CA\,  and $A\subset C$ be a \SCA\, such that $1_{\tilde A}=1_C.$
Denote by $R=\overline{AC},$ the closed  right ideal of $C$ generated by $A.$ 
Then

(1) $H_A\hspace{-0.03in}\bullet\hspace{-0.03in} C=\{\{b_n\}\in H_C: b_n\in R\}.$

(2) If $C$ is a $W^*$-algebra and $e_\af\nearrow 1_C,$ where 
$\{e_\af\}$ is an approximate identity for $A,$ 
then
$$(H_A\hspace{-0.03in}\bullet\hspace{-0.03in} C)^\sharp=H_C^\sharp.$$

\end{prop}

\begin{proof}
To see (1), we first note that $A\hspace{-0.03in}\bullet\hspace{-0.03in} C=R$  as Hilbert $C$-modules.
Hence $A^{(n)}\hspace{-0.03in}\bullet\hspace{-0.03in} C=R^{(n)}.$ 
Clearly, $H_A\hspace{-0.03in}\bullet\hspace{-0.03in} C\subset H_C.$
We note that $\{\{r_n\}\in H_C: r_n\in R\}$ is closed in $H_C.$ 
  Since 
both $\cup_n A^{(n)}\hspace{-0.03in}\bullet\hspace{-0.03in} C$ and $\cup_{n=1}^\infty R^{(n)}$ are  dense 
in  $\{\{r_n\}\in H_C: r_n\in R\},$  and $\cup_n A^{(n)}\hspace{-0.03in}\bullet\hspace{-0.03in} C$ is dense 
$H_A\hspace{-0.03in}\bullet\hspace{-0.03in} C,$ 
 we obtain that 
\beq
\{\{r_n\}\in H_C: r_n\in R\}=H_A\hspace{-0.03in}\bullet\hspace{-0.03in} C.
\eneq
This proves  (1).

For (2) we 
 may assume 
that $A\subset C\subset B(X),$ 
where $X$ is a Hilbert space, $1_C={\rm id}_X,$  and $C(X)=X.$
Otherwise, we replace $X$ by $1_C(X).$

Claim 1: $\overline{R(X)}=C(X)=X.$  Since $e_\af\nearrow 1_C={\rm id}_X,$
for any $v\in X,$ $e_\af(v)\to v.$ This proves the claim.

Claim 2: $R^\sharp=C,$ where $R^\sharp$ is the dual of Hilbert $C$-module $R$
(as we now assume that $C$ is a $W^*$-algebra).

Let $f\in R^\sharp.$
Then $f(e_\af)r=f(e_\af r)\to f(r)$ for all $r\in R$ in norm as $e_\af r\to r$ in norm. 
Hence $f(e_\af)r(v)\to f(r)(v)$ for all $r\in R$ and $v\in X.$ 
Define $T$ on $R(X)$ by $T(r(v))=\lim_\af f(e_\af)r(v)$  for all $v\in X$ and $r\in R.$ One checks that it is 
a well-defined linear map on $R(X).$ Moreover, 
$\|T\|\le \sup\{\|f(e_\af)\|: \af\}\le \|f\|.$ 
Since, by Claim 1,  $\overline{R(X)}=X,$  $T$ extends uniquely a bounded 
operator (denote by $T$ again) on $X.$ Moreover, $f(e_\af)$ converges to $T$ on $X.$ 
Since $C$ is  closed in the weak operator topology, $T\in C.$ 
Moreover $Tr(v)=f(r)(v)$ for all $v\in X.$ 
It follows that $Tr=f(r)$ for all $r\in  R.$
%Denote $f(1)=T.$ Then $f(r)=f(1)r$ for all $r\in R.$ 

For each $c\in C,$ define $f_c\in R^\sharp$ by
\beq
f_c(r)=c^*r\rforal  r\in R. 
\eneq
For the above $T,$
we note that $f_{T^*}(r)=T r$ for all $r\in R.$
%Then, by what has just been proved,  
Hence the map $c\to f_c$ is surjective.
To see it is injective,  
suppose that $c^*r=0$ for all $r\in R.$  Then
\beq
c^*e_\af c=0\rforal \af.
\eneq
Since $c^*e_\af c\nearrow c^*c,$ this implies that $c^*c=0.$ Thus the map
$c\mapsto f_c$ is injective which extends the identity map on $R.$
It follows that $R^\sharp=C.$ Claim 2 is proved.

By Claim 2, we obtain that $((A^{(n)})\hspace{-0.03in}\bullet\hspace{-0.03in} C)^\sharp=C^{(n)}.$
By (1),  $(A^{(n)})\hspace{-0.03in}\bullet\hspace{-0.03in} C$ is a direct summand of $H_A\hspace{-0.03in}\bullet\hspace{-0.03in} C.$  
Hence we may write  $((A^{(n)})\hspace{-0.03in}\bullet\hspace{-0.03in} C)^\sharp\subset (H_A\hspace{-0.03in}\bullet\hspace{-0.03in} C)^{\sharp}.$
Together with (1), we obtain that
\beq
H_A\hspace{-0.03in}\bullet\hspace{-0.03in} C\subset H_C\subset (H_A\hspace{-0.03in}\bullet\hspace{-0.03in} C)^{\sharp}. 
\eneq
Note $H_C$ is a Hilbert $C$-submodule of the self-dual Hilbert $C$ module $(H_A\hspace{-0.03in}\bullet\hspace{-0.03in} C)^{\sharp}.$
%On the other hand $H_C^\sharp$ is also a self-dual Hilbert $C$-module
%which extends $H_C$ and hence extends $H_A\hspace{-0.03in}\bullet\hspace{-0.03in} C.$
It follows from Lemma \ref{ext} that 
\beq
(H_A\hspace{-0.03in}\bullet\hspace{-0.03in} C)^\sharp \subset H_C^\sharp\subset (H_A\hspace{-0.03in}\bullet\hspace{-0.03in} C)^{\sharp}. 
\eneq
consequently,
$H_C^\sharp=(H_A\hspace{-0.03in}\bullet\hspace{-0.03in} C)^\sharp.$
\end{proof}

\begin{NN}
{\rm Note that, if $A$ is unital, $H_A\hspace{-0.03in}\bullet\hspace{-0.03in} C=H_C.$}
\end{NN}

 From the above discussion, we obtain the following:
 
 \begin{lem}\label{L21}
 Let $A$ be a \CA, $H_n=(A^{**})^{(n)}$ 
 and $P_n: H_{A^{**}}^\sharp \to H_n$ be the projection. 
 
 (1) Let $S\subset H_{A^{**}}^\sharp$ be a bounded subset. 
 Then, for any $f\in A^*$ and $x\in H_{A^{**}}^\sharp,$ 
 \beq
 && \lim_{n\to\infty}\sup\{| f(\la P_n(x), y\ra )-f(\la x, y\ra)|: y\in S\}=0\andeqn\\
 && \lim_{n\to\infty}\sup\{| f(\la y, P_n(x)\ra )-f(\la y, x\ra)|: y\in S\}=0.
 \eneq
%both  uniformly on $S.$

(2)  Moreover
 \beq
 \lim_{n\to\infty}| f(\la P_n(x), P_n(x)\ra )-f(\la x, x\ra)|=0\rforal x\in H_{A^{**}}^{\sharp}\andeqn f\in A^*.
 \eneq
  \end{lem}

\begin{proof}
\iffalse
Let $x=\{a_n\}\in H_A^\sharp.$
We compute  that
\beq
| f(\la P_n(x), P_n(x)\ra )-f(\la x, x\ra)|=\sum_{k=n+1}^\infty f(a_k^*a_k).
\eneq
By what has be discussed in \ref{DAsharp}, 
\beq
\lim_{n\to\infty}\sum_{k=n+1}^\infty f(a_k^*a_k)=0.
\eneq
This proves  the first part of the lemma.
\fi

Set $M=\sup\{\|y\|: y\in S\}+1.$ 
%$be such that, for any $y\in S,$
%$\|y\|\le M.$
Let $f$ be a positive linear functional in $A^*$ and $x=\{a_n\}\in H_{A^{**}}^\sharp.$ 
For each $y=\{b_n\}\in S,$ 
\beq
|f(\la P_n(x), y\ra)-f(\la x, y\ra)|
&=& |\sum_{k=n+1}^\infty f(a_k^*b_k)|\\
&\le& (\sum_{k=n+1}^\infty f(a_k^*a_k))^{1/2}(\sum_{k=n+1}^\infty
f(b_k^*b_k))^{1/2}\\
&\le & \|f\|\|y\|  (\sum_{k=n+1}^\infty f(a_k^*a_k))^{1/2}\\
&\le & M\|f\| (\sum_{k=n+1}^\infty f(a_k^*a_k))^{1/2}.
\eneq
By what has been discussed in \ref{DAsharp}, 
\beq
\lim_{n\to\infty}(\sum_{k=n+1}^\infty f(a_k^*a_n))^{1/2}=0.
\eneq
Thus, for this $f$ and $x,$ 
$|f(\la P_n(x), y\ra)-f(\la x, y\ra)|$ converges uniformly on $S.$
Almost the identical estimates shows that $|f(\la y, P_n(x)\ra)-f(\la y, x\ra)|$ converges uniformly on $S.$

Since any $f\in A^*$ can be written as a linear combination of four positive linear functionals
in $A^*,$  the first part of the statement holds.

For the second part, we note that, for any $f\in A^*$ and $x\in H_{A^{**}}^\sharp,$
by the first part of the lemma (since $\|P_n(x)\|\le \|x\|$),
\beq
\lim_{n\to\infty}|f(\la P_n(x), P_n(x)\ra)-f(\la x, P_n(x)\ra)|=0.
\eneq
%By \ref{PElsharp}, w
We also have 
\beq
\lim_{n\to\infty}| f(\la P_n(x), x\ra)-f(\la x, x\ra)|=0.
\eneq
Hence the  second part of the lemma also follows.
\end{proof}

The following are two easy facts:
We present here for convenience.

\begin{lem}\label{Telambda}
Let $A$ be a \CA.

(1) Let $H$ be a Hilbert $A$-module, and 
$\{E_\lambda\}$ be an approximate identity for $K(H).$
Suppose $T\in K(H)^{**}$ is a non-zero  positive element.
Then there is $\lambda_0$ such that
\beq
E_\lambda TE_\lambda \not=0\rforal \lambda\ge \lambda_0.
\eneq

(2) Let $T\in K(H_A)^{**}$ be a non-zero positive element and $P_n: H_A\to H_n=A^{(n)}$ be the projection ($n\in \N$). 
Then, there exists $n_0\in \N$ such that
\beq
P_nTP_n\not=0\rforal n\ge n_0.
\eneq

\end{lem}

\begin{proof}
Let $f\in A^*$ be a positive linear functional. 
Then 
\beq
|f(T^{1/2}(1-E_\lambda))|^2\le f(T)f((1-E_\lambda)^2)\le f(T)f((1-E_\lambda))\to 0.
\eneq
It follows that $f(T^{1/2}E_\lambda)\to f(T^{1/2})$ for all positive linear functionals, whence 
for all $f\in A^*.$
Since $T^{1/2}\not=0,$ for some $\lambda_0,$ $T^{1/2}E_\lambda\not=0$ for all $\lambda\ge \lambda_0.$
It follows that 
$$
E_\lambda TE_\lambda\not=0
$$
for all $\lambda\ge \lambda_0.$
This proves (1).

There are several easy proofs for (2). Let us use part (1).
Choose an approximate identity  $\{e_\af\}$ for $A.$
Let $\lambda=(\af, n)$ and $\lambda_1=(\bt_1, n)\le \lambda_2=(\bt_2, m)$
if $\bt_1\le \bt_2$ and $n\le m.$ 
Define 
\beq
E_{\bt, n}={\rm diag}(\overbrace{e_\bt, e_\bt,...,e_\bt}^n,0,...).
\eneq
Then $\{E_{\bt, n}\}$ forms an approximate identity for $K(H_A)\cong A\otimes {\cal K}.$
Let $T\in K(H)^{**}_+$ be a non-zero positive element.
By (1), there is $\bt_0$ and $n_0\in \N$ such that 
\beq
E_{\bt, n} TE_{\bt,n}\not=0\rforal (\bt, n)\ge (\bt_0, n_0).
\eneq
Hence $\|T^{1/2}E_{\bt, n}^2T^{1/2}\|=\|E_{\bt, n} TE_{\bt,n}\|\not=0$ for all $(\bt, n)\ge (\bt_0, n_0).$
Since
\beq
T^{1/2}P_nT^{1/2}\ge T^{1/2}E_{\bt, n}^2T_{1/2}\not=0,
\eneq
we have $T^{1/2}P_nT^{1/2}\not=0.$ It follows that
\beq
P_nTP_n\not=0\rforal n\ge n_0.
\eneq
\end{proof}

\begin{lem}\label{LMT-1}
Let $A$ be a  \CA\, and $H$ be a countably generated Hilbert $A$-module.
Then the \hm\,  $\Psi$ from 
$K(H)^{**}$ into  $B(H^\sim)$  (give by Proposition \ref{s2-adp}) is injective.
%as $W^*$-algebra. 
\end{lem}

\begin{proof}
%It suffices to show that $\Psi$ is surjective. 

Let $H_n=A^{(n)}=\{(a_1, a_2,...,a_n): a_j\in A\}$ be the Hilbert $A$-module  whose inner product is 
defined by
$\la x, y\ra=\sum_{j=1}^n a_j^*b_j,$ where $x=\{a_j\}_{1\le j\le n}$ and $y=\{b_j\}_{1\le j\le n}.$ 
One identifies $K(H_n)$ with $M_n(A).$ 

Claim: The map $\Psi:  K(H_n)^{**}\to B(H_n^\sim)$ is a $W^*$-isomorphism.

Since we identify  $K(H_n)$ with $M_n(A),$  we have $K(H_n)^{**}=M_n(A^{**}).$

%It is easy to see that $H_n\hspace{-0.03in}\bullet\hspace{-0.03in} A^{**}=(A^{**})^{(n)}$  which is self-dual. 
By (2) of Proposition \ref{HAB} (and Claim 2 of the proof), $(H_n\hspace{-0.015in}\bullet\hspace{-0.015in} A^{**})^\sharp=(A^{**})^{(n)}.$ 
So $H_n^\sim=(A^{**})^{(n)}.$  Note that $B(H_n^\sim)=M_n(A^{**}).$
One then easily checks that $\Psi: K(H_n)^{**}\to B(H_n^\sim)$ is a $W^*$-isomorphism.
This proves the claim.

  Let us consider  the map $\Psi_{H_A}: K(H_A)^{**}\to B(H_A^\sim)$  given by Proposition \ref{PBL}.
  %the case that $H=H_A.$
  Put $T\in K(H)^{**}_+\setminus \{0\}$ 
  
  By (2) of Lemma \ref{Telambda},  there exists $n_0\in \N$ such that
  $P_nTP_n\not=0.$  
  Recall that $H_n$ is a direct summand of $H_A.$ Hence 
  by the claim and applying  Proposition \ref{Ladd}, we conclude that  
  have $\Psi_{H_A}(P_nTP_n)\not=0$ for all $n\ge n_0.$ 
  There must be  an element $x\in H_n$
  such that
  \beq
  \la \Psi_{H_A}(P_nTP_n)(x), x\ra \not=0.
  \eneq
  It follows that $\la \Psi_{H_A}(T)x, x\ra\not=0.$    Hence 
  $\Psi_{H_A}(T)\not=0.$ This implies that ${\rm ker}\Psi_{H_A}=\{0\}.$  
  
  In general, since $H$ is countably generated, by Kasparov's absorbing theorem (Theorem 2 of \cite{Ka}),
  we may write $H_A=H\oplus H^\perp.$  
  % Denote by $\Psi_{H_A}: K(H_A)^{**}\to 
  %B(H_A^\sim)$ the \hm\, given by Proposition \ref{PBL}.
   To show $\Psi$ is injective, 
  let $T\in B(H)^{**}$ be a non-zero element.
  Then $K(H)^{**}=PK(H_A)^{**}P,$ where $P: H_A\to H$ is the projection.
  Hence $PTP=T$ in $K(H_A)^{**}.$
  We have shown that $\Psi_{H_A}(PTP)\not=0.$
  Since, by Lemma \ref{Ladd}, $\Psi(T)=P\Psi_{H_A}(T)P|_{H^\sim}\not=0.$
  This implies that $\Psi$ is injective.
  
\end{proof}

\begin{lem}\label{LMT-2}
Let $A$ be a \CA\, and $H$ be a countably generated Hilbert $A$-module.
Then there is an isomorphism $\Psi$ from 
$K(H)^{**}$ onto $B(H^\sim)$ as $W^*$-algebras. 
\end{lem}

\begin{proof}
By Lemma \ref{LMT-1} (and by Proposition \ref{PBL}), 
it suffices to show that $\Psi$ is surjective. 
%
\iffalse
Let $H_n=A^{(n)}=\{(a_1, a_2,...,a_n): a_j\in A\}$ be the Hilbert $A$-module 
defined by
$\la x, y\ra=\sum_{j=1}^n a_j^*b_j,$ where $x=\{a_j\}_{1\le j\le n}$ and $y=\{b_j\}_{1\le j\le n}.$ 
One identifies $K(H_n)$ with $M_n(A).$ It follows that $K(H_n)^{**}=M_n(A^{**}).$

It is easy to see that $(H_n\otimes A^{**}/N)^-=(A^{**})^{(n)}$  which is self-dual. 
So $H_n^\sim=(A^{**})^{(n)}.$  Note that $B(H_n^\sim)=M_n(A^{**}).$
One then easily checks that $\Psi: K(H_n)^{**}\to B(H_n^\sim)$ is a $W^*$-isomorphism.
\fi
%
Let us first  consider  the case $H=H_A$ (even though $H_A$ is not countably generated when
$A$ is not $\sigma$-unital).
By the end of \ref{DHsim} (see also  Remark 3.9 (and 3.10) of \cite{Pa}), 
to show that $T\in B(H^\sim)=B(H_{A^{**}}^\sharp)$ is in $\Psi(K(H_A)^{**}),$
it suffices to show  that, for any $\ep>0,$ any finite subsets $X\subset 
H_{A^{**}}^\sharp$ and a finite subset ${\cal F}\subset  A^*,$ 
there exists $S\in K(H)^{**}$ such that
\beq
|f(\la \Psi(S)(x), y\ra )-f(\la T(x), y\ra)|<\ep\rforal x, y\in X\andeqn f\in {\cal F}.
\eneq

For any $T\in B(H_A^\sim)=B(H_A^\sharp),$ 
\beq
|f(\la P_nTP_n(x), y\ra -f(\la T(x), y\ra)|
&\le & |f(TP_n(x), P_n(y)\ra)-f(\la T(x), P_n(y)\ra)|
\\
&&+
|f(\la T(x), P_n(y)\ra )-f(\la T(x), y\ra)|
\eneq
for any $x, y\in H_{A^{**}}^\sharp$ and $f\in A^*.$
However,  $\|P_n(y)\|\le \|y\|$ for all $n\in \N.$ By (1) of Lemma \ref{L21},
\beq
|f(TP_n(x), P_n(y)\ra)-f(\la T(x), P_n(y)\ra)|\to 0,
\eneq
 and 
by (2) of Lemma \ref{L21},
\beq
|f(\la T(x), P_n(y)\ra )-f(\la T(x), y\ra)|\to 0.
\eneq
It follows that 
\beq
\lim_{n\to\infty}|f(\la P_nTP_n(x), y\ra -f(\la T(x), y\ra)|=0
\eneq
for all $x, y\in H_A^\sharp$ and $f\in A^*.$

We then choose $n_0\in \N$ such that for all $n\ge n_0$ (recall $P_n$ is a projection)
\beq\label{LTM-2-01}
|f(\la P_nTP_n(x), P_n(y)\ra)-f(T(x), y\ra)|<\ep
\eneq
for all $x, y\in X$ and $f\in {\cal F}.$

Now fix $n\ge n_0.$ Then we have 
$P_n(x), P_n(y)\in (H_n)^\sim$ for all $x, y\in X,$  and $P_nTP_n\subset B(H_n^\sim).$ 
By the claim  for $H_n$ in the proof of Lemma \ref{LMT-1},  
we obtain an element $S\in K(H_n)^{**}$
such that $\Psi_n(S)=(P_nTP_n)|_{H_n^\sim},$
where $\Psi_n: K(H_n)^{**}\cong M_n(A^{**}) \to B(H_n^\sim)=M_n(A^{**})$ is the isomorphism 
given  by the claim.
Note, by Lemma \ref{Ladd}, that $\Psi(S)=P_n\Psi(S)=\Psi(S)P_n=\Psi_n(S).$
\iffalse
we have an element $S\in K(H_n)$ such that
\beq
|f(\la \Psi(S)P_n(x), P_n(y)\ra)-f(\la P_nTP_n(x), P_n(y))|<\ep/2\rforal x\in X, y\in Y
\eneq
and $f\in {\cal F}.$
Note that $P_n\Psi(S)P_n=\Psi(S).$   Hence 
\fi
Hence 
it  follows that,
% (applying also \eqref{LTM-2-01}), 
for all $x, y\in X$ and $f\in {\cal F}$ (and $n\ge n_0$),
\beq\nonumber
\hspace{-0.4in}|f(\la \Psi(S)(x), y\ra)-f(\la T(x), y\ra)|&=&
|f(\la P_n\Psi(S)P_n(x), y\ra)-f(T(x), y\ra)|\\\nonumber
 &=&
|(f(\la P_n\Psi(S)P_n(x), P_n(y)\ra)- f(\la P_nTP_n(x), P_n(y))|\\\nonumber
&&+|f(\la P_nTP_n(x), P_n(y)\ra)-f(T(x), y\ra)|\\
&&<0+\ep=\ep.
\eneq
%
\iffalse
\beq\nonumber
\hspace{-0.4in}|f(\la \Psi(S)(x), y\ra)-f(\la T(x), y\ra)|&=&|(f(\la \Psi(S)P_n(x), P_n(y)\ra)- f(\la P_nTP_n(x), P_n(y))|\\\nonumber
&&+|f(\la P_nTP_n(x), P_n(y)\ra)-f(T(x), y\ra)|\\
&&<0+\ep=\ep.
\eneq
\fi
As mentioned above, this implies that $\Psi$ is surjective. 

For a general  countably generated Hilbert $A$-module $H,$
by Kasparov's  absorbing theorem (Theorem 2 of \cite{Ka}), 
%$H\oplus H_A\cong H_A.$
we may write $H_A=H\oplus H^\perp.$
By \ref{Ladd}, $H_A^\sim=H^\sim\oplus (H^\perp)^\sim.$ 
Let $S\in B(H^\sim)\setminus \{0\}.$ 
Define $T\in B(H_A^\sim)$ by $T|_{H^\sim}=S$ and $S|_{(H^\perp)^\sim)}=\{0\}.$
We have shown that there is $L\in B(H_A)^{**}$ such 
that $\Psi_{H_A}(L)=S.$ 
Then $PSP=S,$ by \ref{Ladd}, $\Psi(L)=P\Psi_{H_A}(L)P|_{H^\sim}=T.$
Hence $\Psi$ is surjective.
%
\iffalse
Let $\Psi_{H_A}, \Psi_{H}, \Psi_{H^\perp}$ be 
the \hm s give by ? for $H_A,$ $H$ and $H^\perp,$ respectively. 
By ?, $\Psi_{H_A}=
\Psi_H\oplus \Psi_{H^\perp}.$ 
From what have been proved, $\Psi$ is surjective. It follows that both $\Psi_H$ is surjective.
\fi
\end{proof}

\begin{thm}\label{MMT}
Let $A$ be a \CA\, and $H$ be a Hilbert $A$-module.
Then there is an isomorphism $\Psi$ (given by Proposition \ref{PBL}) from 
$K(H)^{**}$ onto $B(H^\sim)$ as $W^*$-algebras.
Moreover,
\beq
\Psi|_{B(H)}=\tilde \Psi_0.
\eneq 
\end{thm}

\begin{proof}
By Proposition \ref{PBL}, it suffices to show that $\Psi$ is bijective. 
If $K(H)$ is unital, by Proposition 2.8 of \cite{BL}, $H$ is finitely generated. The theorem then follows from Lemma \ref{LMT-2}.
So we will assume that $K(H)$ is not unital.

Let $\{E_\lambda\}$ be an approximate identity for $K(H)$  and 
$H_\lambda=\overline{E_\lambda(H)}.$ 
Then $K(H_\lambda)=\overline{E_\lambda K(H)E_\lambda}$
is $\sigma$-unital.   By Proposition 3.2 of \cite{BL}, 
$H_\lambda$ is countably generated. 

Denote by $P_\lambda: H^\sim\to H_\lambda^\sim $ the projection given by \ref{ext}
and $\Psi_\lambda: K(H_\lambda)^{**}\to B(H_\lambda^\sim)$ be the map given by Proposition \ref{PBL}.

To see $\Psi$ is injective, let $T\in K(H)^{**}_+\setminus \{0\}.$
It follows from Lemma \ref{Telambda} that $E_\lambda TE_\lambda\not=0$ for all $\lambda\ge \lambda_0$
and  some $\lambda_0.$
Since $H_\lambda$ is countably generated, by Lemma \ref{LMT-2}, 
$\Psi_\lambda(E_\lambda TE_\lambda)\not=0$ (for $\lambda\ge \lambda_0$).
By Proposition \ref{Ladd}, 
\beq
\Psi(E_\lambda T E_\lambda)|_{H_\lambda ^\sim}=\Psi_\lambda(E_\lambda T E_\lambda).
\eneq
It follows   that $\Psi(E_\lambda T E_\lambda)|_{H_\lambda ^\sim}\not=0$
for all $\lambda\ge \lambda_0.$  For $\lambda\ge \lambda_0,$ there are $x, y\in H_\lambda$
such that
\beq
\la \Psi(T)(E_\lambda(x)), E_\lambda(y)\ra=\la \Psi(E_\lambda T E_\lambda)(x), y\ra\not=0.
\eneq
Hence $\Psi(T)\not=0.$ This shows that $\Psi$ is injective.

To see that $\Psi$ is surjective, 
let $L\in B(H^\sim).$ 
Since, by Proposition \ref{PBL}, $\Psi(K(H)^{**})$ is weak*-closed in 
$W^*$-algebra $B(H^\sim),$ it suffices to show the following:
for any $\ep>0,$ any finite subsets $X, Y\subset H^\sim,$ and 
finite subset ${\cal F}\subset A^*,$ there exists 
$T\in K(H)^{**}$ such that
\beq\label{MT-101}
|f(\la \Psi(T)(x), y\ra)-f(L(x), y\ra)|<\ep\rforal x\in X,\,\, y\in Y\andeqn f\in {\cal F}
\eneq
(see the last part of \ref{DHsim}).
We now fix $\ep,$ $X, Y$ and ${\cal F}.$
By Proposition \ref{PBL} (since $E_\lambda\nearrow 1_{K(H)^{**}}$),   
\beq
\lim_\lambda f(\la x, \Psi(E_\lambda)(y)\ra)=\lim_\lambda f(\la \Psi(E_\lambda)(x), y\ra)=f(\la x, y\ra)
\eneq
for all $x, y\in H^\sim$ and $f\in A^*.$
It follows that there is $\lambda_0$ such that, for all $\lambda\ge \lambda_0,$ 
\beq
&&|f(\la \Psi(E_\lambda)(x), L^*(y)\ra)-f(\la x, L^*(y)\ra)|<\ep/2,\\
{\rm or}\,\,\,&& |f(\la L\Psi(E_\lambda)(x), y\ra)-f(\la L(x), y\ra)|<\ep/2
%%&& |f(\la x, L^*(y)\ra)-f(\la E_\lambda(x), L^*(y)\ra|<\ep/3,\\
%%{\rm or}\,\,\,&&
%%|f(\la L(x), y\ra)-f(\la LE_\lambda(x), y\ra)|<\ep/3
\eneq
 for all $x\in X,$ $y\in Y$ and $f\in {\cal F}.$
 (The proof would be shorter if we know \\
$\lim_\lambda f(\la \Psi(E_\lambda)L\Psi(E_\lambda)(x), y\ra)=f(\la L(x),y\ra).$)
 However, we may also assume that, for fixed $\lambda_0,$ 
 there is $\lambda_1\ge \lambda_0$ such that
 \beq
 &&|f(\la L\Psi(E_{\lambda_0})(x), \Psi(E_\lambda)(y)\ra)-f(\la L\Psi(E_{\lambda_0})(x), y\ra)|<\ep/2,\\
{\rm or}\,\,\, &&|f(\la \Psi(E_\lambda)L\Psi(E_{\lambda_0})(x), y\ra)-f(\la L\Psi(E_{\lambda_0})(x), y\ra)|<\ep/2
 \eneq
 for all $x\in X,$ $y\in Y$ and $f\in {\cal F},$ and $\lambda\ge \lambda_1.$ 
 It follows that, for all $x\in X,$ $y\in Y$ and $f\in {\cal F},$  if $\lambda\ge \lambda_1,$
 \beq
 &&\hspace{-0.6in}|f(\la L(x), y\ra)-f(\la \Psi(E_\lambda)L\Psi(E_{\lambda_0})(x), y\ra)|\\
 &&\le |f(\la L(x), y\ra)-f(\la L\Psi(E_{\lambda_0})(x), y\ra)|\\
&&\hspace{0.4in}+ |f(\la L\Psi(E_{\lambda_0})(x), y\ra)-f(\la \Psi(E_\lambda)L\Psi(E_{\lambda_0})(x), y \ra)|\\\label{MT-1003}
&&<\ep/2+\ep/2=\ep.
 \eneq
 
 Fix $\lambda\ge \lambda_1\ge \lambda_0.$
 Then $H_\lambda=\overline{E_\lambda(H)}\supset H_{\lambda_0}.$
 We also note that $\Psi(E_\lambda)L\Psi(E_{\lambda_0})|_{H_\lambda^\sim}\in B(H_\lambda).$
 Since $H_\lambda$ is countably generated, by Lemma \ref{LMT-2}, there is $T_\lambda\in K(H_\lambda)^{**}$
 such that 
 \beq\label{MT-1001}
 \Psi_\lambda(T_\lambda)=\Psi(E_\lambda)L\Psi(E_{\lambda_0})|_{H_\lambda^\sim}.
 \eneq
 However, by Proposition \ref{Ladd}, 
 \beq\label{MT-1002}
P_\lambda\Psi(T_\lambda)P_\lambda|_{H_\lambda^\sim}= \Psi(T_\lambda)|_{H_\lambda^\sim}=\Psi_\lambda(T_\lambda).
 \eneq
 Fix $\lambda\ge \lambda_1\ge \lambda_0.$ Then, for any $x\in X, y\in Y$ and $f\in A^{**},$
 by \eqref{MT-1002}, \eqref{MT-1001} and \eqref{MT-1003},
 \beq\nonumber
\hspace{-0.5in} |f(\la \Psi(T_\lambda)(x), (y)\ra)-f(\la L(x), y\ra)|
 &=&|f(\la \Psi(T_\lambda)P_\lambda(x), P_\lambda(y)\ra)-f(\la L(x), y\ra)|\\\nonumber
 &=& |f(\la \Psi(E_\lambda)L\Psi(E_{\lambda_0})P_\lambda(x), P_\lambda(y)\ra)-f(\la L(x), y\ra)|\\\nonumber
 &=& |f(\la P_\lambda\Psi(E_\lambda)L\Psi(E_{\lambda_0})P_\lambda(x), y\ra)-f(\la L(x), y\ra)|\\\nonumber
 &=&|f(\la \Psi(E_\lambda)L\Psi(E_{\lambda_0})(x), y\ra)-f(\la L(x), y\ra)|<\ep.
 \eneq 
 As mentioned above, this implies that $\Psi$ is surjective.
 
 \end{proof}
 
 \begin{cor}
 Let $A$ be a $W^*$-algebra and $H$ be a Hilbert $A$-module.
 Then $F: K(H)^{**}\to B(H^\sharp),$  the map given by Proposition \ref{s2-adp}, is a surjective map.
 \end{cor}

 \begin{proof}
 Consider the pair $A$ and $A^{**}$ and $H^\sim =(H\hspace{-0.03in}\bullet \hspace{-0.03in} A^{**})^\sharp.$
 By Corollary 4.3 of \cite{Pa}, $H^\sim=B(H, A^{**}),$ the $A^{**}$-module of all bounded $A^{**}$-valued $A$-module maps 
 from $H$ into $A^{**}.$ It follows that $H^\sharp\subset H^\sim$ as an $A$-submodule. 
 It then follows from Proposition \ref{Psmall} that $H^\sharp\hspace{-0.03in}\bullet \hspace{-0.03in} A^{**}\subset H^\sim$
 as Hilbert $A^{**}$-modules.  Then, by applying Lemma \ref{ext}, 
 $$
 (H^\sharp\hspace{-0.03in}\bullet \hspace{-0.03in} A^{**})^\sharp\subset  H^\sim.
 $$
 However, $H\hspace{-0.03in}\bullet \hspace{-0.03in} A^{**}\subset H^\sharp\hspace{-0.03in}\bullet \hspace{-0.03in} A^{**}.$ By applying Lemma \ref{ext} again, we obtain 
 $$
H^\sim= (H\hspace{-0.03in}\bullet \hspace{-0.03in} A^{**})^\sharp\subset (H^\sharp\hspace{-0.03in}\bullet \hspace{-0.03in} A^{**})^\sharp.
 $$
 Hence $ (H^\sharp\hspace{-0.03in}\bullet \hspace{-0.03in} A^{**})^\sharp= H^\sim.$
 Denote by $\tilde \Psi: K(H)^{**}\to B(H^\sim)$ the isomorphism given by Theorem \ref{MMT} and 
 by $\tilde \Psi_{H^\sharp}: B(H^\sharp)\to B((H^\sharp\hspace{-0.01in}\bullet \hspace{-0.01in} A^{**})^\sharp)=B(H^\sim)$
  the map given by Theorem \ref{MMT}, respectively. 
  
  Now let  $T\in B(H^\sharp).$  Then, by applying Theorem \ref{MMT},  we obtain  $a\in K(H)^{**}$ 
  such that $\tilde \Psi(a)=\tilde \Psi_{H^\sharp}(T).$ 
  It follows that 
  \beq
  \tilde \Psi(a)|_{H^\sharp}=T.
  \eneq
  Since $a\in K(H)^{**},$ there exists a net $\{a_\af\}$ in $K(H)$ such that 
  $a_\af\to a$ in the weak*-topology.
  Therefore, by Proposition \ref{PBL}, for any $f\in A^*$ and  any $\xi, \zeta \in H^\sim,$
  \beq
  \lim_\af f(\la (\tilde\Psi(a)-\tilde\Psi(a_\af))(\xi), \zeta\ra)=0.
  \eneq
  On the other hand, by Proposition \ref{s2-adp}, for any $g\in A_*$   and any $x, y\in H,$
  \beq
  \lim_\af g(\la (F(a)-a_\af)(x), y\ra)=0.
  \eneq
  Hence 
  \beq
  g(\la (F(a)-\tilde\Psi(a))(x),y\ra)=0\rforal x, y\in H\andeqn g\in A_*.
  \eneq
  Since $\tilde \Psi(a)|_{H^\sharp}=T,$ 
  we actually have 
  \beq\label{gggg}
  g(\la (F(a)-T)(x),y\ra)=0\rforal x,\, y\in H\andeqn g\in A_*.
  \eneq
Note that  $F(a),\, T\in B(H^\sharp).$   So $F(a)(x), T(x)\in H^\sharp$  for all $x\in H.$ 
It follows that 
\beq
\la (F(a)-T)(x),y\ra\in A\rforal x,\,y\in H.
\eneq
By \eqref{gggg}, 
$$
\la (F(a)-T)(x), y\ra=0\rforal x, y\in H.
$$
Hence $F(a)=T.$ In other words, $F$ is surjective.
 \end{proof}

 \section{A Kaplansky density theorem in Hilbert modules}
 
 As mentioned in the introduction,  in this section we study the density 
 of $H$ in $H\hspace{-0.03in}\bullet\hspace{-0.03in} A^{**}.$ 

 \begin{df}\label{DTs}
 Let  $X$ be a Hilbert space and $A\subset B(X)$ be a \SCA\, of $B(H).$
 Let $M={\bar A}^{^{SOT}},$  the strong operator closure of $A,$ 
  and $H$ be a Hilbert $A$-module. 
  Recall, by Proposition \ref{Psmall}, $H\hspace{-0.03in}\bullet\hspace{-0.03in} M$ is the smallest 
  Hilbert $M$-module containing $H$ as Hilbert $A$-module.
  We consider  the question  how large $H$ in $H\hspace{-0.03in}\bullet\hspace{-0.03in} M$ as a submodule. 
 % Then, by section 4 of \cite{Pa}, $(H\hspace{-0.03in}\bullet\hspace{-0.03in} M)$ becomes a Hilbert $M$-module
%  which extends $H$ in a natural way (see also \ref{DHsim}). 
  
 Let $\ep>0,$ 
 %$Y\subset H\hspace{-0.03in}\bullet\hspace{-0.03in} M$ be a finite subset
 and $V$ be a finite subset of $X.$
  For each $\xi \in H\hspace{-0.03in}\bullet\hspace{-0.03in} M,$ define  
 \beq
 N_{\xi, \ep, V}=\{z\in  H\hspace{-0.03in}\bullet\hspace{-0.03in} M: \|\la \xi-z,\xi-z\ra(v)\|<\ep,\,\,\, v\in V\}.
 %\{z\in H\hspace{-0.03in}\bullet\hspace{-0.03in} M: \|\la \xi -z, x\ra(v)\|<\ep,\,\, x\in Y\andeqn v\in V\}.
 \eneq
 Let ${\cal T}_s$ be the topology generated by 
 $N_{\xi, \ep,  V}$ for all $\xi\in H\hspace{-0.03in}\bullet\hspace{-0.03in} M,$ $\ep\in \R_+\setminus \{0\},$
 any finite subset $Y\in H\hspace{-0.03in}\bullet\hspace{-0.03in} M$ and any finite subset $V\subset X.$
In other words, in ${\cal T}_s,$ a net $\{z_\af\}$ converges to $\xi$ in $H\hspace{-0.03in}\bullet\hspace{-0.03in} M$
if and only if 
\beq
\lim_\af \|\la \xi-z_\af, \xi-z_\af\ra(v)\|=0\rforal v\in X.
\eneq

In the special case that $X=H_U$
is the Hilbert space corresponding to the universal representation $\pi_U$ of $A$ 
and $M=A^{**},$ we use ${\cal T}_{su}$ 
for the topology 
\iffalse
 Let $A$ be be a \CA\, and $H$ be a Hilbert $A$-module.
 Let $\ep>0,$ $Y\subset H\hspace{-0.03in}\bullet\hspace{-0.03in} A^{**}$ be a finite subset
 and $V$ be a finite subset of $H_U,$ the Hilbert space  corresponding to the 
 universal representation $\pi_U$ of $A.$
 For each $\xi \in H\hspace{-0.03in}\bullet\hspace{-0.03in} A^{**},$ define  
 \beq
 N_{\xi, \ep, X, V}=\{z\in H\hspace{-0.03in}\bullet\hspace{-0.03in} A^{**}: \|\pi_U(\la \xi -z, x\ra)(v)\|<\ep\,\, x\in X\andeqn v\in V\}.
 \eneq
 Let ${\cal T}_s$ be the topology
 \fi
  generated by 
 $N_{\xi, \ep, V}$ for all $\xi\in H\hspace{-0.03in}\bullet\hspace{-0.03in} A^{**},$ $\ep\in \R_+\setminus \{0\},$
 %any finite subset $Y\subset H\hspace{-0.03in}\bullet\hspace{-0.03in} A^{**}$ 
 and any finite subset $V\subset H_U.$
 \end{df}
 
 It is easy to see that $H$ is dense in $H\hspace{-0.03in}\bullet\hspace{-0.03in} M$ in topology ${\cal T}_s.$
 However, to be more useful, we will show in Theorem \ref{LP1} that 
 the unit ball of $H$ is dense in the unit ball of $H\hspace{-0.03in}\bullet\hspace{-0.03in} M$ in ${\cal T}_s,$
 a Kaplansky style density theorem.
 
 \begin{lem}\label{4LL+1}
 Suppose that $x\in H\hspace{-0.03in}\bullet\hspace{-0.03in} M$ and $\{x_\af\}\subset H\hspace{-0.03in}\bullet\hspace{-0.03in} M$ is a bounded net.
 Then 
 $x_\af\to x$ in ${\cal T}_s$ if and only if, for any $v\in X,$  
 \beq
\lim_\af \sup\{\|\la y, x_\af-x\ra(v)\|: y\in H\hspace{-0.03in}\bullet\hspace{-0.03in} M,\,\,\|y\|\le 1\}=0.
 \eneq
Moreover,  if $x_\af\to x$ in ${\cal T}_s,$
%if one of the above holds, 
then, for any $f\in M_*,$ 
\beq
\lim_\af \sup\{|f(\la y, x_\af-x\ra)|: y\in H\hspace{-0.03in}\bullet\hspace{-0.03in} M, \,\,\|y\|\le 1\}=0.
\eneq
 \end{lem}
 
 \begin{proof}
 Suppose that $x_\af\to x$ in ${\cal T}_s.$
 We have (see (ii) of Proposition 2.3 of \cite{Pa}), for any $y\in H\hspace{-0.03in}\bullet\hspace{-0.03in} M,$  
 \beq
\la x_\af-x, y\ra \cdot  \la y, x_\af-x\ra\le \|y\|^2 \la x_\af-x,x_\af-x\ra.
 \eneq
 Then, for any $v\in X,$ and any $y\in  H\hspace{-0.03in}\bullet\hspace{-0.03in} M$ with $\|y\|\le 1,$ 
 \beq
 \|\la y, x_\af-x\ra(v)\|^2&=&
\la  \la x_\af-x, y\ra \cdot \la y, x_\af-x\ra v, v\ra_X \\
 &\le & \|y\|^2\la \la x_\af-x, x_\af-x\ra v, v\ra_X\\
 &\le & \|\la x_\af-x, x_\af-x\ra v\|\|v\|\to 0
 \eneq
 (where $\la \cdot, \cdot\ra_X$ is the inner product in the Hilbert space $X$).
 Conversely, let $K=\sup_\af\{\|x_\af\|+\|x\|\}+1.$ 
 Then 
 \beq
 \|\la x_\af-x, x_\af-x\ra (v)\|\le K\sup\{ \|\la y, x_\af-x\ra (v)\|: y\in H\hspace{-0.03in}\bullet\hspace{-0.03in} M,\,\, \|y\|\le 1\}\to 0
 \eneq
  
  For the ``Moreover" part of the lemma, 
  suppose that
  $\la x_\af-x, x_\af-x\ra\to 0$ in  the strong operator topology. 
  Then it converges in the weak operator topology. However $\{\la x_\af-x, x_\af -x\ra \}$ is bounded, 
  this also implies that it converges to zero in $\sigma$-weak topology and in weak*-topology.
  In other words, 
  \beq
  \lim_\af f(\la x_\af-x, x_\af-x\ra)=0\rforal f\in M_*. 
  \eneq
  Let $f\in M_*$ be a positive normal functional.
  Then,  $f(\la \cdot, \cdot\ra)$ defines a pseudo inner product on $H\hspace{-0.03in}\bullet\hspace{-0.03in} M.$
  Hence,
   for any $y\in H\hspace{-0.03in}\bullet\hspace{-0.03in} M,$
  we have, by Cauchy-Bunyakovsky-Schwarz inequality,  
  \beq
  |f(\la y, x_\af-x\ra)|^2\le f(\la y, y\ra) f(\la x_\af-x, x_\af-x\ra)\le \|f\|^2\|y\|^2 f(\la x_\af-x,x_\af-x\ra).
  \eneq 
  Thus
  \beq
  \lim_\af \sup\{|f(\la y, x_\af-x\ra)|: y\in H\hspace{-0.03in}\bullet\hspace{-0.03in} M,\,\, \|y\|\le 1\}=0.
  \eneq
  
 \end{proof}
 
 \begin{lem}\label{4LL1}
 Let $X$ be a Hilbert space, $A\subset B(X)$ be a \SCA\, and 
 $M=\overline{A}^{{SOT}}$ such that ${\rm id}_X\in M.$  
 %and 
 %$e_\af\nearrow 1_M={\rm id}_X,$ 
 %where 
% $\{e_\af\}$ is an approximate identity for $A.$
% Let $H$ be a Hilbert $A$-module.
 Then the unit ball of $H_A$ is dense in  the unit ball  of 
 $H_A\hspace{-0.03in}\bullet\hspace{-0.03in} M$ in  ${\cal T}_{s}.$
\end{lem}

\begin{proof}
Let $\xi\in H_A\hspace{-0.03in}\bullet\hspace{-0.03in} M$ with $\|\xi\|\le 1.$
We will show that  there is a net $\{x_\af\}\in H$ such that 
$\|x_\af\|\le \|\xi\|$ and $\lim_\af \|\la x_\af-\xi, x_\af-\xi\ra (v)\|=0$
for all $v\in X.$   From the inequality
$$
\|\la x_\af-\xi, x_\af-\xi\ra (v)\|\le \|\la x_\af-\xi,x_\af, \xi\ra^{1/2}\|\|\la x_\af-\xi,x_\af, \xi\ra^{1/2}(v)\|
\le 2\|\|\la x_\af-\xi,x_\af, \xi\ra^{1/2}(v)\|,
$$
we conclude that it is enough to show that there is  net $\{x_\af\}\in H$ such that 
$\|x_\af\|\le \|\xi\|$ and $\lim_\af \|\la x_\af-\xi, x_\af-\xi\ra^{1/2} v\|=0$
for all $v\in X.$ 
%
\iffalse
Suppose that we have found a net $\{x_\af\}$ in $H$
such that  $\la x_\af-\xi, x_\af-\xi\ra^{1/2}$ converges in the strong operator topology. Then,
since real continuous functions on self-adjoint elements are continuous in the strong operator 
topology, $\la x_\af-\xi, x_\af-\xi\ra$ converges in the strong operator topology.
\fi
Therefore 
it suffices to show that, for any $\ep>0$ 
%any finite subset $Y\subset H_A\hspace{-0.03in}\bullet\hspace{-0.03in} M$ 
and 
any finite subset $V\subset X,$ there exists $z\in H$ with $\|z\|\le 1$ such 
that
\beq
\|(\la \xi-z, \xi-z\ra)^{1/2}(v)\|<\ep\rforal v\in V.
\eneq
%Put $M:=1+\max\{\|y\|: y\in Y\}.$  
To simplify notation, we may also assume 
that $\|v\|\le 1$ for all $v\in V.$

%Let us consider first the case that
%$H=H_A.$ 
Denote by $R={\overline{AM}},$ the closed right ideal of $M$ generated by $A.$
Note, by Proposition \ref{HAB}, 
\beq
H_A\hspace{-0.03in}\bullet\hspace{-0.03in} M=\{\{b_n\}\in H_B: b_n\in R\}
.
\eneq
We write $\xi=\{b_n\}\in H_A\hspace{-0.03in}\bullet\hspace{-0.03in} M.$ 
There exists $n_0\in \N$ such that, for all $n\ge n_0,$ 
\beq
\|\sum_{k=n}^\infty b_n^*b_n\|<\ep/2.
\eneq
Fix an integer $n\ge n_0.$
Let  $P_n: H_A\hspace{-0.03in}\bullet\hspace{-0.03in} M\to R^{(n)}=\{(c_1, c_2,...,c_n): c_i\in R\}$ be the projection.
Put 
\beq
S=\begin{pmatrix} b_1 & 0& 0\cdots & 0\\
                                   b_2 & 0& 0\cdots &0\\
                                   \vdots && &\vdots\\
                                   b_n & 0 & 0\dots & 0\end{pmatrix}.
                                   \eneq
                                 
                                   \iffalse
                                   \andeqn
Z=\{\begin{pmatrix} y_1& 0 & 0\cdots &0\\
                              y_2 &0& 0\cdots &0\\
                              \vdots &&& \vdots\\
                              y_{n} &0& 0\cdots &0\end{pmatrix}: (y_1,y_2,...,y_n)\in PY\}
                              \eneq
                              Then
                              \beq
                              \|S\|^2=\|S^*\|^2=\|SS^*\|=\|\sum_{i=1}^n b_i^*b_i\|
                              \eneq
                             for any $z\in Z,$
                              \beq
                              \|z\|^2=\|z^*z\|=\|\sum_{i=1}^n y_i^*y_i\|\le M^2.
                              \eneq
                              \fi
 %  Put $u_v=(v,0,0,\cdots,0)^T$ for each $v\in V$
  % (where $T$  stands for the transpose).                       
            % Note that, f
             For any $v\in V,$ 
             %and $z\in Z,$  
              put 
             \beq
             u_v=\begin{pmatrix} v\\
                      0\\
                      \vdots\\
                      0\end{pmatrix}. 
                      %\,\,\,{\rm Then}\,\,\, 
            % z(u_v)=\begin{pmatrix} y_1(v)\\
               %               y_2(v) \\
                 %             \vdots\\
                 %             y_{n} (v)\end{pmatrix}.
                              \eneq
     By the Kaplansky density theorem,
     there is $L\in M_n(A)$ such that
     \beq\label{4LL-101}
     \|L\|\le \|S\|\andeqn 
     %\|L(z(u_v))-S(z(u_v))\|
     \|L(u_v)-S(u_v)\|<\ep/2
     \eneq
     for all $v\in V.$
     Hence, denoting by $\la \cdot,\cdot\ra_X$  the inner product in $X,$ 
     \beq\label{4LL-101+}
     \la(L-S)^*(L-S)u_v, u_v\ra_X<\ep/2\rforal v\in V.
     \eneq
     % and $z\in Z.$
     Define $q={\rm diag}(1,0,...,0)\in M_n(M).$ 
     Then $S=Sq.$ Replacing $L$ by $Lq,$ we may write 
     \beq
     L=\begin{pmatrix} a_1 & 0& 0 \cdots  & 0\\
                                    a_2 & 0& 0\cdots &0\\
                                   \vdots&&& \vdots\\
                                   a_n & 0 & 0\dots & 0\end{pmatrix},
     \eneq
     where $a_i\in A,$ $i=1,2,...,n.$
     Then 
     \beq\label{SLL-102}
     \|\sum_{i=1}^n a_i^*a_i\|=\|L^*L\|=\|L\|^2\le \|S\|^2=\|\sum_{i=1}^n b_i^*b_i\|\le \|\xi\|^2.
     \eneq
       It follows from \eqref {4LL-101+} that
       \beq
       \la \sum_{i=1}^n(b_i-a_i)^*(b_i -a_i)(v), v\ra_X<\ep/2.
       \eneq  
       Put $x=(a_1, a_2,...,a_n, 0,0, \cdots)\in H_A.$ 
       Then, by \eqref{SLL-102}, $\|x\|\le \|\xi\|,$ and 
       \beq
      \hspace{-0.4in} \la \la \xi-x, \xi-x\ra(v), v\ra_X &=&\la \sum_{i=1}^n(b_i-a_i)^*(b_i-a_i)(v),v\ra_X+\la \sum_{i=n+1}^\infty b_i^*b_i(v),v\ra_X\\
       &<& \ep/2+\|\sum_{i=n+1}^\infty b_i^*b_i\|<\ep/2+\ep/2=\ep.
      % &\le & \ep/2+
     % ( \|\sum_{i=n+1}^\infty b_i^*b_i\|\|\sum_{i=n+1}^\infty y_iy_i^*\|)^{1/2}<\ep/2+\ep/2=\ep.
       \eneq
           In other words, for any $v\in V,$ 
           \beq
           \|(\la \xi-x_\af, \xi-x_\af)^{1/2}(v)\|=\la\la \xi-x, \xi-x\ra(v), v\ra_X<\ep.
           \eneq
           The lemma then follows.
           \end{proof}

 \begin{thm}\label{LP1}
 Let $X$ be a Hilbert space, $A\subset B(X)$ be a \SCA\, and 
 $M=\overline{A}^{{SOT}}$ and ${\rm id}_X\in M.$
 Let $H$ be a Hilbert $A$-module.
 Then the unit ball of $H$ is dense in  the unit ball  of 
 $H\hspace{-0.03in}\bullet\hspace{-0.03in} M$ in  ${\cal T}_{s}.$
\end{thm}

\begin{proof}
Let $\xi\in H\hspace{-0.03in}\bullet\hspace{-0.03in} M$ with $\|\xi\|\le 1.$
%Put $M:=1+\max\{\|y\|: y\in Y\}.$  

Let us first assume that $H$ is countably generated $A$-module.
By Lemma \ref{4LL+1}, it  suffices to show that, for any $\ep>0$  and
%any finite subset $Y\subset H\hspace{-0.03in}\bullet\hspace{-0.03in} M$ and 
any finite subset $V\subset X,$ there exists $z\in H$ with $\|z\|\le 1$ such 
that
\beq
\|\la y,  \xi-z\ra)(v)\|<\ep\rforal y\in H \,\,{\rm with}\,\, \|y\|\le1\andeqn v\in V.
\eneq
To simplify notation, we may also assume 
that $\|v\|\le 1$ for all $v\in V.$

By Kasparov's absorbing theorem (Theorem 2 of \cite{Ka}),
we may write $H_A=H\oplus H^\perp.$ 
It follows that $$H_A\hspace{-0.03in}\bullet\hspace{-0.03in} M=H\hspace{-0.03in}\bullet\hspace{-0.03in} M\oplus H^\perp \hspace{-0.03in}\bullet\hspace{-0.03in} M.$$ 
Define $Q: H_A\to H$ to be the projection.
Then $Q\in L(H_A)=M(K(H_A)).$  We identify $Q$ with $\Psi_0(Q)$
in the sense that $Q\in L(H_M)$ which extends $Q|_{H_A}.$ 
In particular, $H\hspace{-0.03in}\bullet\hspace{-0.03in} M=Q(H_A\hspace{-0.03in}\bullet\hspace{-0.03in} M).$ 

By applying Lemma \ref{4LL1} and Lemma \ref{4LL+1}, we obtain $z\in H_A$ with $\|z\|\le \|\xi\|$
such that
\beq
\|\la y, \xi-z\ra(v)\|<\ep\rforal y\in H\hspace{-0.03in}\bullet\hspace{-0.03in} M,\,\,\|y\|\le 1\andeqn v\in  V.
\eneq
Note $Q(\xi)=\xi$ and $Q(y)=y$ for all $y\in H.$ Put $x=Q(z)\in H.$ 
We have 
\beq 
\hspace{-0.2in} \|\la y, \xi-x\ra(v)\|=\|\la y,Q( \xi)-Q(z)) \ra(v)\|=\|\la Q(y), \xi-z\ra (v)\|=\|\la y, \xi-z\ra(v)\|<\ep.
 \eneq
This proves the case that $H$ is countably generated.

Next we let $H$ be a general Hilbert $A$-module.  We will show that,
for any $\ep>0$  and
%any finite subset $Y\subset H\hspace{-0.03in}\bullet\hspace{-0.03in} M$ and 
any finite subset $V\subset X,$ there exists $z\in H$ with $\|z\|\le 1$ such 
that
\beq
\|\la \xi-z,  \xi-z\ra)(v)\|<\ep\rforal  v\in V.
\eneq
Again, we may also assume 
that $\|v\|\le 1$ for all $v\in V.$

Let $\{E_\lambda\}$ be an approximate identity for $K(H).$
Then, as in the proof of Theorem \ref{MMT}, $H_\lambda=\overline{E_\lambda(H)}$ is countably generated
for each $\lambda.$ 
% Note $H_\lambda=\overline{E_\lambda^{1/2}(H)}.$
%Moreover  $\{E_\lambda^{1/2}\}$ is  also an approximate identity for $K(H).$ 
It follows from Lemma \ref{Lelambda} that there is $\lambda$ such that
\beq\label{LP1-201}
\|\Psi_0(E_\lambda)(\xi)-\xi\|<\ep/4.
%M\andeqn \|\Psi_0(E_\lambda)(y)-y\|<\ep/6M
\eneq
%for all $y\in Y.$
Fix such  a $\lambda.$   Note that, by Proposition \ref{PTid},  $\Psi_0(E_\lambda)(\xi)\in H_\lambda \hspace{-0.03in}\bullet\hspace{-0.03in} M\subset 
H\hspace{-0.03in}\bullet\hspace{-0.03in} M.$ 
Since $H_\lambda$ is countably generated, by the first part of the proof, 
we obtain $x\in H_\lambda$ with $\|x\|\le \|\Psi_0(E_\lambda)(\xi)\|\le \|\xi\|$ such 
that
\beq\label{LP1-202}
\sup\{\|\la y, \Psi_0(E_\lambda)(\xi)-x\ra (v)\|: y\in H\hspace{-0.03in}\bullet\hspace{-0.03in} M,\,\,\|y\|\le 1\}<\ep/4.
%\|\la \Psi_0(E_\lambda)(\xi)-x, \Psi_0(E_\lambda)(y)\ra(v)\|<\ep/3M\rforal y\in Y.
\eneq
Then, applying \eqref{LP1-201} and  then \eqref{LP1-202},
\beq
\hspace{-0.4in}\|\la \xi-x, \xi-x\ra(v)\| &\le &\|\la\xi-x, \xi-\Psi_0(E_\lambda)(\xi)\ra(v)\|+
\|\la \xi-x, \Psi_0(E_\lambda)(\xi)-x\ra(v)\|\\
&<& 2\| \xi-\Psi_0(E_\lambda)(\xi)\|
+2\|\la {\xi-x\over{2}}, \Psi_0(E_\lambda(\xi)-x\ra(v)\|\\
&<& \ep/2+\ep/2=\ep.
\eneq
\end{proof}
 We then obtain the following corollary as a Kaplansky density theorem:
 
  \begin{thm}\label{P1}
 Let $A$ be a \CA\, and $H$ be a Hilbert $A$-module.
 Then the unit ball of $H$ is dense in  the unit ball  of 
 $H\hspace{-0.03in}\bullet\hspace{-0.03in} A^{**}$ in  ${\cal T}_{su}.$
\end{thm}

\section{Closeness of $H$}

Let $H$ be a Hilbert $A$-module., 
Then, by Theorem 6.1 of \cite{BL},  the unit ball  of $H$ is $A$-weakly dense (see Definition 3.3 of \cite{BL}) in the unit ball of $H^\sharp,$ 
i.e., for any $f\in H^\sharp,$ there is a net $\{x_\af\}$ in $H$ with 
$\|x_\af\|\le \|f\|$ such that
$\lim_\af \|\la f-x_\af, y\ra\|=0$ for all $y\in H.$  In the case that $A$ is a $W^*$-algebra, 
$H^\sharp$ is a Hilbert $A$-module. One may ask whether 
one can find the net $\{x_\af\},$ in addition, $\lim_\af \|\la f-x_\af, \xi\ra\|=0$ for all $\xi\in H^\sharp?$

%By Theorem \ref{P1}, it is easy to show that $H$ is dense in $H^\sim$ in ${\cal T}_0,$ the topology 
%defined below (see \ref{DT0}).  A similar question is whether one can replace $x$ in \eqref{dt0}
%{DTw}
%by an element in $H^\sim.$ 

We begin with the following example.
\begin{exm}\label{Ex2}
Let $M$ be a $W^*$-algebra which contains a self-adjoint element $a$ with infinite spectrum.
Then, by the spectral theory, one obtains a sequence of mutually orthogonal non-zero projections 
$p_1, p_2,...,p_n,....$ 
Let $H=H_M$ and 
let $\xi=\{p_n\}\in H_M^\sharp.$ 
Note that $\|\xi\|=\|\sum_{n=1}^\infty p_n\|=1$
(the convergence is in the strong operator topology and weak*-topology of $M$).
We claim that there is {\it no} net $\{x_\af\}$ in $H_M$ 
%with $\|x_\af\|\le 1$ 
such that
\beq
\lim_\af \|\la \xi-x_\af, \xi\ra\|=0.
\eneq
Otherwise, 
there would be  $x\in H_M$ such that
\beq\label{Ex-101}
\|\la \xi-x, \xi\ra\|<1/4.
\eneq
Since $x=\{a_n\}\in H_M,$ there is $N\in \N$ such that
\beq
\|\sum_{N+1}a_n^*a_n\|<(1/16(1+\|x\|))^2.
\eneq
\iffalse
Combining \eqref{Ex-101}, we obtain 
\beq
\|\sum_{i=1}^N (b_i-a_i)^*b_i\|&=&\|\sum_{i=1}^\infty (b_i-a_i)^*b_i-\sum_{i=N+1}^\infty (b_i-a_i)^*b_i\|\\
&<& 1/14+\|
\eneq
\fi
Choose $q=\sum_{n=N+1}^\infty p_n\in  M.$ 
Define $P_N: H_M^\sharp\to M^{(N)}=\{(b_1, b_2,...,b_N): b_i\in M\}$ to be 
the projection. 
Then
\beq
\|\la \xi-P(x), \xi\ra\| &\le &\|\la \xi-x, \xi\ra \|+\|\la (1-P_N)(x), \xi\ra\|\\
&<& 1/4+\|(1-P_N)(x)\|\|\xi\|<1/4+1/16=5/16.
%\andeqn\\
%&&\|\la (1-P_N)(\xi-x), \xi\ra-\la (1-P_N)\xi, \xi\ra\|=\|\la (1-P_N)(x), \xi\ra\|<1/16.
\eneq
On the other hand 
\beq
5/16&\ge& \|\la \xi-P(x), \xi\ra\|\ge \|\la \xi-P(x), \xi\ra q\|\\
  &&=\|(\sum_{N+1}^\infty p_n-\sum_{i=1}^N (p_i-a_i)^* p_i)q\|\\
&=&\|\sum_{N+1}^\infty p_n q\|=1.
\eneq
%\|\la \xi-x, \xi\ra\| &\ge & |\|\la (1-P_N)(\xi-x), \xi\ra\|-\|\la P_N(\xi-x), \xi\ra\||\\
%&=&\|\la (1-P_N)(\xi), \xi\ra-1/6-
%\|\la (1-P_N) (\xi-x), \xi\ra\|-\|P_N(\xi-x), P_n(\xi)\ra\||\\\
%\eneq
A contradiction.  In  other words, the question at the beginning of this section is negative.
This also follows from Corollary \ref{TTcc} below. However, we think that the example above 
might also be  helpful.
%an  illustrated certain inside 
%
\end{exm}

\begin{lem}\label{Lclose1}
Let $A$ be a \CA. 
Suppose that  $\xi\in H_A^\sharp$ and $\{x_\af\}$ is a bounded 
net such that
\beq
\lim_\af\|\xi(x)-x_\af(x)\|=0\tforal x\in H_A
\eneq
and   $\xi(x_\af):=\la \xi, x_\af\ra $ converges in norm. Then  $\xi\in H_A$ and $\la \xi, \xi\ra=\lim_\af\la \xi, x_\af\ra.$
%\beq
%\lim_\af \|\xi-x_\af\|=0\tand \xi\in H_A.
%\eneq
%Then $H_A$ is closed as a sub-$A$-module of $H_A^\sharp$ 
%in ${\cal T}_{nw}.$

\end{lem}

\begin{proof}
Write $\xi=\{b_n\}$ and $x_\af=\{a_{\af, n}\},$
where $\{b_n\}\in H_A^\sharp,$  $a_{\af, n}\in A$
and, for each $\af, $ $\{a_{\af , n}\}\in H_A.$

Put 
\beq
M=1+\sup\{\|x_\af\|: \af\}+\|\xi\|<\infty\andeqn a=\lim_\af\la x_\af, \xi\ra.
\eneq
Note $\xi(x_\af)=\la \xi,x_\af\ra\in A$ for all $\af.$ Hence $a\in A.$ 

Put $P_n: H_A^\sharp\to H_n:=A^{(n)}$ be the projection to the first $n$ copies of 
$A.$ Then $P_n\xi\in H_n\subset H_A.$
It follows that
\beq\label{Lclose1-101}
\lim_\af\la x_\af, P_n(\xi)\ra=\la \xi, P_n(\xi)\ra =\sum_{j=1}^n b_j^*b_j.
\eneq
Fix $f\in A^*.$ Let $\ep>0.$
By Lemma \ref{L21}, since $\{x_\af\}$ is bounded, there is an integer $N\in \N$ such that, for all $n\ge N,$
\beq
|f(\la x_\af, \xi\ra)-f(x_\af, P_n(\xi)\ra)|<\ep/3\rforal \af.
\eneq
Fix any $n\ge N.$
Choose $\af_0$ such that, for all $\af\ge \af_0,$
\beq
&&\|\la x_\af, P_n(\xi)\ra-\sum_{j=1}^n b_j^*b_j\|<\ep/3(1+\|f\|)\andeqn\\
&&\|\la  x_\af, \xi\ra-a\|<\ep/3(1+\|f\|).
\eneq
It follows that, for  all $n\ge N,$ 
\beq
|f(a)-f(\sum_{j=1}^n b_j^*b_j)| &\le& |f(a-\la x_\af, \xi\ra)|+|f(\la x_\af, \xi\ra)-f(x_\af, P_n(\xi)\ra)|\\
&&+\|f\|\|\la x_\af, P_n(\xi)\ra-\sum_{j=1}^n b_j^*b_j\|<\ep/3+\ep/3+\ep/3=\ep.
\eneq
Hence on the state space $S(A)$ of $A,$
\beq
\lim_{n\to\infty} f(\sum_{j=1}^n b_j^*b_j)=f(a).
\eneq
On the compact space $S(A)$ (in the weak *-topology), $\hat{a}(f)=f(a)$ for all $f\in S(A)$ is a continuous 
function, and $\widehat{\sum_{j=1}^n b_j^*b_j}$ is increasing.
By the Dini theorem, $\widehat{\sum_{j=1}^n b_j^*b_j}$ converges uniformly to $\hat{a}$ 
on $S(A).$   It follows that
\beq
\sum_{j=1}^n b_j^*b_j\to a
\eneq
in norm.   This implies that $\xi=\{b_n\}\in H_A$ and $\la \xi, \xi\ra=a=\lim_\af \la \xi, x_\af\ra.$
%
%%%%%%%%%%%%%
\iffalse
Choose $N_1\in \N$ such that 
\beq
\|\sum_{j=n+1}^\infty b_j^*b_j\|<(\ep/2M)^2\rforal n\ge N_1.
\eneq
By \eqref{Lclose1-101}, choose $\af_0$ such that, for all $\af\ge \af_0,$
\beq
\|\la \xi-x_\af,\xi\ra\| <\ep/2.
\eneq
%
Then, for all $\af\ge \af_0,$
\beq
\|\la \xi-x_\af, \xi-x_\af\ra\|&\le& \|\la \xi-x_\af, \xi\ra\|+\|\la \xi-x_\af, (1-P_{N_1})(\xi)\ra\|\\
&<&\ep/2+M(\sum_{j=N_1+1}^\infty b_j^*b_j)^{1/2}<\ep.
\eneq
Hence $\lim_\af\|\xi-x_\af\|=0.$
\fi
\end{proof}

\begin{prop}\label{PKHsharp}
Let $A$ be a \CA\, and $H$ be a Hilbert $A$-module.
Then, for any $T\in K(H),$  one has
$\Psi_0(T)(H^\sharp)\subset H,$
where $\Psi_0$ is Definition \ref{Dmap2}.

\end{prop}

\begin{proof}
Suppose that $T\in F(H)$ 
and 
$
T=\sum_{i=1}^m \theta_{x_i, y_i}
$
for some $x_i, y_i\in H,$ $i=1,2,...,m.$ 
Then, for any $\xi\in H^\sharp,$
\beq
\Psi_0(T)(\xi)=\sum_{i=1}^m x_i\la y_i, \xi\ra=\sum_{i=1}^m \xi(y_i)^*x_i\in H.
\eneq
Since $F(H)$ is dense in $K(H),$ 
this implies that $\Psi_0(T)(H^\sharp)\subset H.$
\end{proof}

\begin{lem}\label{Lfuniform}
Let $A$ be a \CA, $H$ be a Hilbert $A$-module and $\{E_\lambda\}$ 
an approximate identity for $K(H).$ 
Then, for any $\xi\in H^\sim$ and any $f\in A^*,$
\beq
\lim_\af \sup\{f(\la \xi-\Psi_0(E_\lambda)(\xi), y\ra)|: y\in H^\sim, \|y\|\le 1\}=0.
\eneq
\end{lem}

\begin{proof}
By Lemma \ref{Lelambda}, $\{\Psi_0(E_\lambda)\}$ is an approximate identity for 
 $K(H\hspace{-0.03in}\bullet\hspace{-0.03in} A^{**}).$ 
 In the universal representation of  $K(H\hspace{-0.03in}\bullet\hspace{-0.03in} A^{**}),$ 
 $1-\Psi_0(E_\lambda)$ converges to zero in the strong operator topology. 
 Note that $\|1-\Psi_0(E_\lambda)\|\le 1.$ 
 Therefore $(1-\Psi_0(E_\lambda))(1-\Psi_0(E_\lambda))$ also converges to zero in the strong operator topology. Hence it converges to zero in the weak operator topology. Since $\{(1-\Psi_0(E_\lambda))^2\}$
 is bounded, it also converges to zero 
 in the weak*-topology of $K(H\hspace{-0.03in}\bullet\hspace{-0.03in} A^{**}).$ Recall that $(H\hspace{-0.03in}\bullet\hspace{-0.03in} A^{**})^\sharp=H^\sim.$
 It follows from Proposition 
 \ref{s2-adp}, for any $\xi\in H^\sim,$
 \beq\nonumber
\hspace{-0.2in} \lim_\af |f(\la \xi-\Psi_0(E_\lambda)(\xi), \xi-\Psi_0(E_\lambda)(\xi)\ra)|&=&
\lim_\af |f(\la \xi-F\circ \Psi_0(E_\lambda)(\xi), \xi-F\circ \Psi_0(E_\lambda)(\xi)\ra)\\
 &=&\lim_\af
| f(\la (1-F\circ \Psi_0(E_\lambda))^2(\xi), \xi\ra)|=0.
 \eneq
Suppose that $y\in H^\sim$ and $\|y\|\le 1.$
Then, for any positive linear functional $f\in A^*,$
\beq
f(\la \xi-\Psi_0(E_\lambda)(\xi), y\ra)^2&\le& f(\la \xi-\Psi_0(E_\lambda)(\xi), \xi-\Psi_0(E_\lambda)(\xi)\ra)
f(y^*y)\\
&\le &\|f\|f(\la \xi-\Psi_0(E_\lambda)(\xi), \xi-\Psi_0(E_\lambda)(\xi)\ra).
\eneq
It follows that, for any $f\in A^*,$ 
\beq
\lim_\af\sup\{f(\la \xi-\Psi_0(E_\lambda)(\xi), y\ra)|: y\in H^\sim, \|y\|\le 1\}=0.
\eneq
\end{proof}

\begin{thm}\label{Tclosed}
Let $A$ be a \CA\, and $H$ be a Hilbert  $A$-module.
Suppose that $\xi\in H^\sharp$ and there is a bounded net 
$\{x_\af\}$ in $H$ such that
\beq
\lim_\af\| \xi(x)-\la x_\af, x\ra\| =0\tforal x\in  H
\eneq
and $\xi(x_\af):=\la \xi, x_\af\ra $ converges in norm. Then $\xi\in H$ and $\la \xi, \xi\ra=\lim_\af\la \xi,x_\af\ra\in A.$
\end{thm}

\begin{proof}
%Put $a=\lim_\af \la \xi, x_\af\ra.$
%Let $H_1=\xi A+H$ and define 
%\beq
%\la \xi a+h_1,\xi b+h_2\ra =a^*
%\eneq

First let us assume $H$ is countably generated. 
Then, by Kasparov's absorbing theorem (Theorem 2 of \cite{Ka}),
we may write $H_A=H\oplus H^\perp.$
Then $\xi\in H^\sharp\subset H_A^\sharp.$  
By applying Lemma \ref{Lclose1}, we obtain that 
\beq
\xi\in H_A\andeqn \la \xi, \xi\ra=\lim_\af\la \xi, x_\af\ra.
%\lim_\af \|\xi-x_\af\|=0.
\eneq

Since $\xi(x_\af)\in A,$  then $a=\la \xi, \xi\ra\in A.$
Let $P: H_A\to H$ be the projection.
Then $P\in L(H_A).$  Put $\eta=P(\xi)\in H.$ 
Note that $\la P(\xi)-\xi, x\ra=0$ for all $x\in H$ and $\xi\in H^\sharp.$
Hence $\xi=\eta.$ 
%x_\af\in H$ for all $\af$ and $H$ is norm closed, $\xi\in H.$
Therefore this case follows.

In what follows we will work in $H^\sim$ and use the inner product in $H^\sim$
whenever it is convenient.

In general,  let $a=\lim_\af\la\xi, x_\af\ra.$   Since $\la \xi, x_\af\ra=\xi(x_\af)\in A$
for all $\af,$ we have $a\in  A.$ 

Claim:  $a=\la \xi, \xi\ra$ (in the inner product of $H^\sim$).

Let $\{E_\lambda\}$ be an approximate identity for $K(H).$
Let $\ep>0$ and $f\in A^*$ with $\|f\|\le 1.$ 
By applying Lemma \ref{Lfuniform}, have (since $\{\|\xi-x_\af\|\}$ is bounded)
\beq\label{34-1}
\lim_\lambda \big(\sup_\af\{ |f(\la \xi-\Psi_0(E_\lambda)(\xi),\xi- x_\af\ra )|\}\big)=0.
\eneq
Thus, by applying Lemma \ref{Lfuniform} and  \eqref{34-1}, we obtain $\lambda_0$ such that, for all $\lambda\ge \lambda_0,$
\beq\label{920-1}
&&|f(\la \xi-\Psi_0(E_\lambda)(\xi), \xi)|<\ep/3\andeqn\\\label{920-2}
&&|f(\la \xi-\Psi_0(E_\lambda)(\xi), \xi-x_\af\ra )|<\ep/3\rforal \af.
\eneq
Recall that, by Proposition \ref{PKHsharp}, $\Psi_0(E_\lambda)(\xi)\in H.$
Fix any $\lambda\ge \lambda_0,$ choose $\af_0$ such that, for any $\af\ge \af_0,$ 
\beq\label{920-3}
\|\la \xi, x_\af\ra -a\|<\ep/3\andeqn |f(\la \Psi_0(E_\lambda)(\xi), \xi-x_\af\ra)|<\ep/3.
\eneq
Now, by the first inequality of \eqref{920-3}, \eqref{920-2}  and  then the second inequality of \eqref{920-3},
\beq
\hspace{-0.3in}|f(\la \xi,\xi\ra-a)|&< &\ep/3+|f(\la \xi, \xi\ra -\la \xi, x_\af\ra)|\\
&=&\ep/3+|f(\la \xi, \xi-x_\af\ra)|\\
&\le & \ep/3+|f(\la \xi-\Psi_0(E_\lambda)(\xi), \xi-x_\af\ra)|+
|f(\la \Psi_0(E_\lambda)(\xi), \xi-x_\af\ra)|\\
&<& \ep/3+\ep/3+\ep/3=\ep.
\eneq
Since this holds for any $\ep,$ we conclude that
\beq
f(\la \xi, \xi\ra )=f(a)\rforal f\in A^*.
\eneq
By the Hahn-Banach theorem, we obtain that $\la \xi, \xi\ra=a.$
The claim is proved.

There exists $x_1\in \{x_\af\}$ such that
\beq
\|\la \xi-x_\af, x_1\ra\| <1/2\andeqn \|\la x_1,\xi\ra-a\|<1/2.
\eneq
%Let ${\cal F}_1=\{\xi, x_1\}.$ 
Suppose that we have 
found $x_1, x_2, ..., x_n$ such that 
%and 
%${\cal F}_j=\{\xi, x_1,...,x_j\},$ 
%$j=1,2,...,n-1,$ such that
\beq
\|\la \xi-x_j, x_i\ra\|<1/2^j\andeqn \|\la x_j, \xi\ra-a\|<1/2^j,\,\,i=1,2,..., j-1,
\eneq
and $j=1,2,..., n.$   Then choose $x_{n+1}\in \{x_\af\}$ such that
\beq
\|\la \xi-x_{n+1}, x_i\ra\|<1/2^{n+1}\andeqn \|\la x_{n+1}, \xi\ra-a\|<1/2^{n+1},\,\,i=1,2,..., n.
\eneq
Thus, by the induction, we obtain a subsequence $\{x_n\}$  in $\{x_\af\}$ such that
\beq
 \lim_{n\to\infty}\|\la x_n,\xi\ra-a\|=0 \andeqn \lim_{n\to\infty}\|\la \xi-x_n, x_i\ra\|=0\rforal i\in \N.
\eneq
Denote by $H_0$ the Hilbert $A$-submodule generated by $\{x_1, x_2,..., x_n,...,\}.$
%It follows from Lemma 2.13 of \cite{BL} that we may write $x_n=u_n\la x_n,x_n\ra^\bt$
%for some $u_n\in H$ and $0<\bt<1/2,$ $n\in \N.$
%Let $\{e_\lambda\}$ be an approximate identity for $A.$
%Then, for all $\lambda,$ 
%\beq
%x_ne_\lambda\in H_0\andeqn \lim_{\lambda} \|x_ne_\lambda-x_n\|=0
%\eneq
%for all $n\in \N.$ Hence 
In particular, $x_n\in H_0,$ $n\in \N.$ 
Let $\eta=\xi|_{H_0}.$

Now $H_0$ is countably generated and $x_n\in H_0$ such that
\beq
\lim_{n\to \infty}\|\eta(x_n)-a\|=\lim_{n\to\infty}\|\xi(x_n)-a\|=0.
\eneq
Moreover, if $y=\sum_{i=1}^m x_i\cdot a_i,$ where $a_i\in A,$
then
\beq 
\lim_{n\to\infty}\|\eta(y)-\la x_n, y\ra\|=0.
\eneq
Since $\{x_n\}$ is bounded (since $\{x_\af\}$ is bounded),  this implies that
\beq
\lim_{n\to\infty}\|\eta(y)-\la x_n, y\ra\|=0\rforal y\in H_0.
\eneq
Applying what has been proved, we conclude that $\eta\in H_0$ and 
$\lim_{n\to\infty}\la \eta,x_n\ra=\la \eta, \eta\ra=a .$

Considering 
Hilbert $A^{**}$-modules $H_0\hspace{-0.0in}\bullet \hspace{-0.02in} A^{**}\subset H\hspace{-0.0in}\bullet \hspace{-0.02in} A^{**}.$ 
By Lemma \ref{ext},  we obtain  a projection  $P: H^\sim \to H_0^\sim$ 
such that $P|_{H_0\hspace{-0.01in}\bullet \hspace{-0.01in} A^{**}}={\rm id}_{H_0\hspace{-0.01in}\bullet \hspace{-0.01in} A^{**}}.$
Then $\eta=P(\xi).$
Hence, by the claim, 
\beq\nonumber
\hspace{-0.3in}\|(1-P)\xi\|^2&=&\|\la (1-P)(\xi), (1-P)(\xi)\|\le \|\la (1-P)(\xi), \xi\ra\|+\|\la (1-P)(\xi), P(\xi)\ra\|\\
&=&\|\la \xi, \xi\ra-\la P(\xi), \xi\ra\|+0=\|a-\la P(\xi), P(\xi)\ra\|=\|a-\la \eta, \eta\ra\|=0.
\eneq 
In other words, $P(\xi)=\eta=\xi.$  The theorem follows.
\end{proof}

\begin{df}\label{Def}
Let $A$ be a \CA\, and $H$ be a Hilbert $A$-module.  
Then $H^\sharp\subset H^\sim.$

For each $\xi\in H^\sharp,$ $\ep>0$ and   a finite subset $Y\subset H^\sharp,$ 
define
\beq
O_{\xi, \ep, Y}=\{\zeta\in H^\sharp: \| \la \xi-\zeta, y\ra\|<\ep: y\in H^{\sharp}\},
\eneq
where the inner product is taken from $H^\sharp,$ if $H^\sharp$ is a Hilbert 
$A$-module, or from $H^\sim$ (with value in $A^{**}$).

Denote by ${\cal T}_{NW}$ the topology in $H^\sharp$ generated by
$O_{\xi, \ep, Y}$
for all $\xi\in H^\sharp,$ $\ep\in \R_+\setminus \{0\},$ and finite subsets $Y\subset H^\sharp.$
Note that a net $\{\zeta_\af\}$ converges to $\xi$ in $H^\sharp$ in ${\cal T}_{NW},$
if and only if 
\beq
\lim_\af \|\la \xi-\zeta_\af, y\ra\|=0
\eneq
for all $y\in H^\sharp.$
%\rforal y\in H^{\sharp}，
%\eneq
where the inner product is the one defined above.
\end{df}

%\iffalse
\begin{cor}\label{TTcc}
Let $A$ be  a \CA\, and 
%a \SCA\, of $B(X),$ where $X$ is a Hilbert space and 
%$M=\overline{A}^{SOT}$ with $1_M=1_{B(X)}.$
%Suppose that 
$H$ be 
 a Hilbert  $A$-module.
 Then, with ${\cal T}_{NW},$ the unit ball of $H$ is closed in  $H^\sharp.$
 %$(H\hspace{-0.02in} \bullet\hspace{-0.02in} M)^\sharp.$ 
%
\end{cor}
%\fi

\begin{proof}
Let $\xi\in H^\sharp.$  Suppose that there is a net $\{x_\af\}$ in $H$ with 
$\|x_\af\|\le 1$ such that
\beq
\lim_\af \|\la \xi-x_\af, \eta\ra\|=0\rforal \eta\in  H^\sharp,
\eneq
where the inner product is in $H^\sim.$
Then, for each $x\in H,$ $\lim_\af\|\la \xi-x_\af, x\ra\|=0$
and $\{ \xi(x_\af)\}=\{\la \xi, x_\af\ra\}$ converges in norm to $\la \xi, \xi\ra.$ By Theorem \ref{Tclosed}, 
$\xi\in H.$ 
%and $\lim_\af\|\xi-x_\af\|=0.$
\end{proof}

\begin{cor}\label{TTCC-m}
Let $A$ be a monotone complete \CA\, and $H$ be a Hilbert $A$-module.
Then the unit ball of $H$ is closed in $H^\sharp$ in the topology ${\cal T}_{NW},$ where 
we view $H^\sharp$ as a self-dual Hilbert $A$-module. 
\end{cor}

\begin{lem}\label{LHdenseinHM}
 Let $X$ be a Hilbert space, $A\subset B(X)$ be a \SCA\, and 
 $M=\overline{A}^{{SOT}}$ and ${\rm id}_X\in M.$
 Let $H$ be a  
 %countably generated 
 Hilbert $A$-module. Suppose that $\xi\in H\hspace{-0.03in}\bullet \hspace{-0.03in} M$
 and $\la \xi, x\ra\in A$ for all $x\in H.$ 
 % there exists a net $\{x_\af\}\in H$ with $\|x_\af\|\le \|\xi\|$
% such that 
% \beq
% \lim_\af\|\la \xi-x_\af, y\ra\|=0\tforal y\in H,\,\, \la \xi, x \ra \in A\tforal x\in  H\andeqn \af
% \eneq
 %and $ \lim_\af \la \xi, x_\af\ra=\la \xi, \xi\ra$ converges in norm. 
 %
 Then $\xi\in H.$ 
\end{lem}

\begin{proof}
First let us consider the case that $H=H_A.$
Then, by Proposition \ref{HAB}, 
\beq
H_A\hspace{-0.03in}\bullet \hspace{-0.03in} M=\{\{a_n\}: a_n\in \overline{AM}\andeqn
\sum_{k=1}^n a_k^*a_k\,\,{\rm converges\,\, in\,\, norm}\}.
\eneq
Write $\xi=\{b_n\}\in H_A\hspace{-0.03in}\bullet \hspace{-0.03in} M.$
The condition that $\la \xi, x\ra \in A$ for all $x\in H_A$ implies that
$\xi\in H_A^\sharp.$
It follows that $b_n\in A.$ Hence $\xi\in H_A.$

Next, let us assume that $H$ is countably generated. 
Let $\xi\in H\hspace{-0.03in}\bullet \hspace{-0.03in} M$ and $\la \xi, x\ra\in A$ 
for all $x\in H.$
By Kasparov's absorbing theorem, we may write $H_A=H\oplus H^\perp.$
It follows from what has been proved that $\xi\in H_A.$ 
Let $P: H_A\to H$ be the projection. Then $P(\xi)\in H.$
However, $\la \xi-P(\xi), x\ra=0$ for all $x\in H.$ 
For any $y\in H^\perp,$ Since $\xi\in H\hspace{-0.03in}\bullet \hspace{-0.03in} M,$
$\la \xi, y\ra=0$ for all $y\in H.$ Hence $\xi=P(\xi)\in H.$ 

In general, since $\xi\in H\hspace{-0.03in}\bullet \hspace{-0.03in} M,$
there are $x_{n,i}\in H,$ $i=1,2,...,k(n),$ $b_{n,i}\in M,$ $i=1,2,...,k(n),$  $n\in \N$
such that
\beq
\lim_{n\to\infty}\|\xi-\sum_{i=1}^{k(n)} x_{n,i}\hspace{-0.01in} \bullet \hspace{-0.01in}b_{n,i}\|=0.
\eneq

Let $H_0$ be the Hilbert $A$-submodule generated by $\{x_{n,i}: 1\le i\le k(n),\, n\in \N\}.$
Then $\xi\in H_0\hspace{-0.02in} \bullet \hspace{-0.02in} M$ and $\xi|_{H_0}\in H_0^\sharp$
as $\la \xi, h\ra\in A$ for all $h\in H_0\subset H.$
From what has just  been proved, $\xi\in H_0\subset H.$ 
\end{proof}

We end this section with the following result.

\begin{thm}\label{S5MT}
Let $A$ be a \CA\, and $H$ be a Hilbert $A$-module.
Then the unit ball of $H$ is closed in $H^\sim$ in the  topology ${\cal T}_{NW}$ of $H^\sim=(H\hspace{-0.03in}\bullet \hspace{-0.03in} A^{**})^\sharp.$ 
\end{thm}

\begin{proof}
Let  $\{x_\af\}$ be a net in the unit ball of $H$ and $\xi\in H^\sim$
such that
\beq
\lim_\af \|\la \xi-x_\af, \zeta\ra\|=0\rforal \zeta\in H^\sim.
\eneq
Since $H^\sim= (H\hspace{-0.03in}\bullet \hspace{-0.03in} A^{**})^\sharp$ 
and $H\subset H\hspace{-0.03in}\bullet \hspace{-0.02in} A^{**},$ by applying Corollary \ref{TTCC-m}, 
we conclude that $\xi\in H\hspace{-0.03in}\bullet \hspace{-0.03in} A^{**}.$

We also  have, for all $y\in H,$
\beq
\lim_\af \|\la \xi-x_\af, y\ra\|=0.
\eneq
Since $\la x_\af, y\ra\in A,$ it follows that $\la \xi, y\ra\in A.$
By Lemma \ref{LHdenseinHM}, $\xi\in H.$
\end{proof}

\section{A Kaplansky style density theorem in the self-dual Hilbert modules}

In the last section, we show that   $H$ is closed in  $H^\sharp$ and in $H^\sim$ 
in the topology ${\cal T}_{NW}$ of $H^\sharp,$  and  that of $H^\sim,$  respectively.
In this section, however,  we will show that $H$ is dense in $H^\sim$ in a weaker topology.
%it could be still dense in a much weaker sense. 
In fact, 
by Theorem \ref{P1}, it is easy to show that $H$ is dense in $H^\sharp$ in ${\cal T}_0,$ the topology 
defined below (see \ref{DT0}).  A similar question is whether one can replace $x$ in \eqref{dt0}
%{DTw}
by any element in $H^\sharp.$ 

\begin{df}\label{DT0}
Let $A$ be a $W^*$-algebra and $H$ be a Hilbert $A$-module.

 Let $\ep>0$ and $Y\subset H$ and ${\cal F}\subset A_*$ be finite subsets.
 Let $\xi\in H^\sharp.$ 
 Denote 
 \beq\label{dt0}
 O_{\xi, \ep, Y, {\cal F}}=\{\zeta\in H^\sharp: |f(\la \xi-\zeta, x\ra)|<\ep,\,\, x\in Y,\,\, f\in {\cal F}\}\subset H^\sharp.
 \eneq
 %where $\la \cdot,\cdot\ra$ is the inner product of $H^\sim.$
 Let ${\cal T}_0$  be the topology of $H^\sharp$ generated by the subsets $O_{\xi, \ep, Y, {\cal F}}.$
 %\end{df}
 \iffalse
 By Theorem \ref{P1}, it is easy to see that $H$ is actually dense in $H^\sim$ in ${\cal T}_0.$
 However, one may further ask whether one may replace $X$ by a finite subset in $H^\sim.$
 Thus we also consider the following topology in $H^\sim.$
 \fi
 
 %\begin{df}\label{DTw}
  Let $\ep>0$ and $Y\subset H^\sharp $ and ${\cal F}\subset A_*$ be finite subsets.
 Let $\xi\in H^\sharp.$ 
 Denote 
 \beq
 O_{\xi, \ep, Y, {\cal F}}=\{\zeta\in H^\sharp: |f(\la \xi-\zeta, x\ra)|<\ep,\,\, x\in Y\}\subset H^\sharp.
 \eneq
 Let ${\cal T}_w$  be the topology of $H^\sharp$ generated by the subsets $O_{\xi, \ep, Y, {\cal F}}.$

 In fact,  by Proposition 3.8  of \cite{Pa} and the definition before it, ${\cal T}_w$ is the weak*-topology of $H^\sharp$ as %we view $H^\sharp$
 %=(H\hspace{-0.03in}\bullet\hspace{-0.03in} A^{**})^\sharp$
 %which is 
 a conjugate space. So a natural question is whether $H$ is dense in
 $H^\sharp$ in ${\cal T}_w.$ To be more useful (perhaps may not be used twice on Sundays-- -cf. 2.3.4. of \cite{Pedbook}),
we will also prove a Kaplansky style  density theorem as Theorem \ref{S4-mt}. 

Let us also consider another topology.
Let $\ep>0,$ $\xi\in H^\sharp$   and ${\cal F}\subset A_*$ be a finite subset.
Denote by
\beq
O_{\ep, \xi, {\cal F}}=\{\zeta\in H^\sharp: |f(\la \xi-\zeta, \xi-\zeta\ra )|<\ep,\,\, f\in {\cal F}\}.
\eneq
Let ${\cal T}_{ws}$ be the topology generated by $O_{\ep, \xi, {\cal F}}$
for all $\ep>0,$ $\xi\in H^\sharp$ and finite subsets ${\cal F}\subset A_*.$
Note that ${\cal T}_{ws}$ is stronger than ${\cal T}_w$ which is stronger than 
${\cal T}_0.$
 \end{df}

 \begin{lem}\label{S3-L1}
 Let $X$ be a Hilbert space and $A\subset B(X)$ be 
 a \SCA.  Suppose that $M=\overline{A}^{SOT}$ with ${\rm id}_X\in M$ 
% Let $A\subset B(X)$ $ be a \CA\, and 
and $b=\{b_k\}\in H_{M}^\sharp.$
 %For each $n,$ 
 There is a net $a_\af=\{(a_{1,\af},a_{2, \af},...,a_{n, \af},...\}\in H_A$  such that
  \beq\label{S3-L1-01}
&& \|\sum_{j=1}^\infty a_{j, \af}^*a_{j, \af}\|^{1/2}\le \|b\|\tand\\\label{S3-L1-02}
 &&\lim_\af f(\sum_{j=1}^\infty (b_j-a_{j, \af})^*(b_j-a_{j, \af}))=0
 \eneq
 for all $f\in M_*.$
 \end{lem}
 
 \begin{proof}
 Let $Y=l^2(X),$ the Hilbert space 
 direct sum of countably many copies of $X.$ 
 %Let $\pi_U: A\to H_U$ be the universal representation of $A.$
 %We identify $\pi_U(A)''$ with $A^{**}.$ 
%Consider $X^{(n)}=\overbrace{X\oplus X\oplus\cdots \oplus X}^n$ and 
%$M_n(A)\to B(X^{(n)}$ the corresponding representation of $M_n(A).$
Let ${\bar b}=(c_{i,j})\in B(Y),$ 
where $c_{i,1}=b_i,$ $i\in \N,$ and $c_{i,j}=0,$ if $j\ge 2$ (see \eqref{bounded}). 
Denote by $P_n: Y\to  X^{(n)}$ the projection, where 
$X^{(n)}$ is the direct sum of (first) $n$ copies of $X.$
Let $\ep>0$ and $V\in L^2(X)$ be a finite subset.
Then there is $n_0\in \N$ such that
\beq
\|(1-P_{n_0})(v)\|<\ep/2(1+\|b\|)\rforal  v\in V.
\eneq
There is $d\in M_{n_0}(A)$ such that
%by the Kaplansky density theorem, such that 
 %$\|d\|\le \|{\bar b}\|$ and 
\beq
\|({\bar b}-d)(P_{n_0}(v))\|<\ep/4\rforal v\in V.
\eneq
%Viewing $d$ as a bounded linear operator on $Y,$
We have
\beq
\|({\bar b}-dP_{n_0})(v)\|&\le & \|({\bar b}-dP_{n_0})(1-P_{n_0})(v)\|+\|({\bar b}-d)P_{n_0}(v)\|\\
&=& \|{\bar b}(1-P_{n_0})(v)\|+\ep/4
<\ep\rforal v\in V.
\eneq 
Let $B$ be the self-adjoint algebra of those bounded operators on $Y$ which can be expressed as 
infinite matrices with entries in $A$ and all are zero except finitely many of them. Then, by what has been proved, we conclude that, in the strong operator topology (of $B(Y)$), 
operator ${\bar b}$ is in the closure of operators in 
$B$ in the strong operator topology.

Then, by the Kaplansky density theorem, there is a net $\{d_\af\}\in B$
with $\|d_\af\|\le \|{\bar b}\|$ such that
\beq
\lim_\af \|({\bar b}-d_\af)v\|=0\rforal v\in Y.
\eneq
Since $\{\|{\bar b}-d_\af\|\}$ is bounded, we also have
\beq
\lim_\af \|({\bar b}-d_\af)^*({\bar b} -d_\af)v\|=0\rforal v\in Y.
\eneq
We further note that
\beq
\|{\bar b}\|^2=\|({\bar b})^*{\bar b}\|=\|\sum_{j=1}^\infty b_j^*b_j\|\le \|b\|.
\eneq

%Let $q_1={\rm diag}(1,0,0,...)$ be the projection on the first copy of $X.$
Then 
\beq\label{S3-L-101}
\lim_\af \|({\bar b}-d_\af)^*({\bar b}-d_\af)P_1v\|=0\rforal v\in Y.
\eneq
Note ${\bar b}P_1={\bar b}.$ 
Let $d_\af'=d_\af P_1=(d_{i,j,\af}),$ 
where $d_{i, j, \af}=0$ if $j\ge 2.$   Put $a_{j, \af}=d_{1, j, \af},$ $j\in \N.$
%One computes that, with $v=(v_1, v_2, ...)\in Y,$ 
Then, for all $n\in \N,$ 
\beq
\|\sum_{j=1}^n a_{j, \af}^*a_{j, \af}\|\le \|(d_\af')^*d_\af'\|=\|d_\af'\|^2\le \|{\bar b}\|^2\le \|b\|^2.
%\sum_{j=1}^n \la a_{j,\af}^*a_{j,\af}(v_1), v_1\ra. 
\eneq
Put $a_\af=\{a_{j, \af}\}.$ Since $d_\af\in B,$  for each $\af,$ there are only finitely many
of $a_{j, \af}$ which are not zero. Hence $a_\af\in  H_A.$  Then $\|a_\af\|\le \|b\|.$
Thus \eqref{S3-L1-01} holds.
%%%%%%%%%%%%%%%%%%%%
\iffalse
It follows that, for any $\af$ and any $v_1\in H_U$ with $\|v_1\|=1,$ 
\beq
\sum_{j=1}^n \la a_{j,\af}^*a_{j,\af}(v_1), v_1\ra\le \|d_\af'\|^2
\eneq
Therefore, for any state $f\in A^*,$
\beq
f(\sum_{j=1}^n a_{j,\af}^*a_{j, \af})\le \|d_\af'\|^2\le \|b\|.
\eneq
Hence 
\beq
\|\sum_{j=1}^n a_{j,\af}^*a_{j, \af}\|\le \|b\|.
\eneq
\fi
On the other hand,  by \eqref{S3-L-101},
\beq\label{S3-L-102}
\lim_\af \|({\bar b}-d_\af')^*({\bar b}-d_\af')P_1v\|=0.
\eneq
Let $h\in X.$  By \eqref{S3-L-102},
\beq
\lim_\af \,\,\,\|\sum_{j=1}^\infty (b_j-a_{j,\af})^*(b_j-a_{j,\af})h\|=0.
\eneq
In other words, $\sum_{i=1}^\infty (b_j-a_{j,\af})^*(b_j-a_{j,\af})=\la b-a_\af, b-a_\af\ra \to 0$ in the strong operator topology.
However, 
\beq
\|\sum_{j=1}^n (b_j-a_{j,\af})^*(b_j-a_{j,\af})\|=\|({\bar b}-d_\af')\|^2\le (\|\bar b\|+\|d_\af\|)^2\le 4\|b\|^2.
\eneq
Therefore $\sum_{i=1}^n (b_j-a_{j,\af})^*(b_j-a_{j,\af})\to 0$ in the $\sigma$-weak operator topology
and hence in weak*-topology (see, for example, 4.6.13 of \cite{Pedb2}).  Therefore \eqref{S3-L1-02}  holds.
%%%%%%%%%%%%%%%
\iffalse
In other words,
\beq
\lim_\af \,\sum_{i=1}^n \la (b_j-a_{j, \af})^*(b_j-a_{j, \af})(h), h\ra=0.
\eneq
For any $h_1, h_2,...,h_m\in H_U,$ we then  have 
that
\beq
\lim_\af \,\sum_{k=1}^m \sum_{i=1}^n \la (b_j-a_{j, \af})^*(b_j-a_{j, \af})(h_k), h_k\ra=0.
\eneq
 It follows that, for any positive  linear functional $f\in A^*,$
 \beq
 \lim_\af \,\sum_{i=1}^n \la f((b_j-a_{j, \af})^*(b_j-a_{j, \af}))=0.
 \eneq
 \fi
 %%%%%%%
 %The lemma then follows.
 \end{proof}

 \begin{lem}\label{S3-L2}
 Let $A\subset B(X)$ be a \SCA\,  
 %such that $e_\af\nearrow {\rm id}_X,$ where $\{e_\af\}$ is an approximate identity,
 and let $M={\bar A}^{SOT}$ with $1_X\in M.$
  Suppose that  $H$ is  a countably generated Hilbert $A$-module.
 Then $H$ is dense in $(H\hspace{-0.03in}\bullet\hspace{-0.03in} M)^\sharp$ in the following sense:
 For any  $\xi\in (H\hspace{-0.03in}\bullet\hspace{-0.03in} M)^\sharp,$ there is a net $x_\af\in H$ with 
 $\|x_\af\|\le \|\xi\|$ such that
 \beq\label{3SL-3-10001}
 \lim_\af \sup\{|f(\la \xi-x_\af, \zeta\ra)|: \zeta\in (H\hspace{-0.03in}\bullet\hspace{-0.03in} M)^\sharp, \|\zeta\|\le 1\}=0\rforal  f\in M_*.
 \eneq
 \end{lem}

 \begin{proof}
 \iffalse
 For any positive normal  linear functional  $f$ of $M,$
 $$
 f(\la \zeta, \xi-x_\af\ra)=\overline{f(\la \xi-x_\af, \zeta\ra)}.
 $$
 Therefore it suffices to show that, for any 
 $\xi\in (H\hspace{-0.03in}\bullet\hspace{-0.03in} M)^\sharp,$ there is a net $\{x_\af\}$ in $H$ such that
 \beq\label{3SL-3-10001}
  \lim_\af \sup\{|f(\la\zeta,  \xi-x_\af\ra)|: \zeta\in (H\hspace{-0.03in}\bullet\hspace{-0.03in} M)^\sharp, \|\zeta\|\le 1\}=0\rforal  f\in A^*.
 \eneq
 \fi
Let us first prove this  for $H=H_A,$ even though when $A$ is not $\sigma$-unital,
 $H_A$ is not countably generated. 
 Lemma \ref{S3-L1} provides a net $\{x_\af\}$ in $H_A$ 
 with $\|x_\af\|\le \|\xi\|$ such that
 \beq
 \lim_\af f(\la \xi-x_\af, \xi-x_\af\ra)=0\rforal f\in M_*.
 \eneq
 Recall that, for any positive linear functional $f,$
 $H_M^\sharp\times H_M^\sharp\to \R$ defined 
 by $[x, y]_f=f(\la x, y\ra)$ (for all $x, y\in H_M^\sharp$) is a pseudo inner product.
 Therefore, by the Cauchy-Bunyakovsky-Schwarz inequality,
 \beq
 f(\la x, y\ra)^2\le f(\la x, x\ra)f(\la y, y\ra)\rforal x, y\in H_M^\sharp.
 \eneq
  It follows that, for any positive normal linear functional $f,$
 \beq\nonumber
\sup\{ |f(\la \xi-x_\af, \zeta \ra)|: \zeta\in H_M^\sharp, \|\zeta\|\le 1\}^2&\le& \sup\{ f(\la \zeta, \zeta\ra )f(\la \xi-x_\af, \xi-x_\af\ra)
:\zeta\in H_M^\sharp,\, \|\zeta\|\le 1\}\\
&=&\|f\|f(\la \xi-x_\af, \xi-x_\af\ra)\to 0.
 \eneq
 Thus we proved \eqref{3SL-3-10001} holds for $H=H_A.$

 Now let $H$ be a countably generated Hilbert $A$-module. 
 Then, by Kasparov's absorbing theorem, we may write $H_A=H\oplus H^\perp.$
 Hence $H_A\hspace{-0.03in}\bullet\hspace{-0.03in} M=H\hspace{-0.03in}\bullet\hspace{-0.03in} M\oplus (H^\perp \hspace{-0.03in}\bullet\hspace{-0.03in} M).$
 It follows that $H_M^\sharp=(H_A\hspace{-0.03in}\bullet\hspace{-0.03in} M)^\sharp=
(H\hspace{-0.03in}\bullet\hspace{-0.03in} M)^\sharp\oplus (H^\perp \hspace{-0.03in}\bullet\hspace{-0.03in} M)^\sharp.$
 %(H^\perp)^\sharp.$
 Let $P: H_M^\sharp\to (H\hspace{-0.03in}\bullet\hspace{-0.03in} M)^\sharp$ be the projection such that $P|_H={\rm id}_H.$ 
 Then,   by what has been proved for $H_A,$ 
 %since the lemma holds for $H_A,$ 
 there is   a net $y_\af\in H_A$ such 
 that $\|y_\af\|\le \|\xi\|$ and, for any $f\in M_*,$
 \beq
 \lim_\af \sup\{ f(\la \xi-y_\af,\zeta\ra): \zeta\in H_M^\sharp, \|\zeta\|\le 1\}=0.
 \eneq 
 Put $x_\af=P(y_\af)\in (H\hspace{-0.03in}\bullet\hspace{-0.03in} M)^\sharp.$ 
 Note that $P(\xi)=\xi.$ Then, for any $f\in M_*,$ 
 \beq\nonumber
 &&\hspace{-0.6in} \lim_\af \sup\{ f(\la \xi-x_\af,\zeta\ra): \zeta\in (H\hspace{-0.03in}\bullet\hspace{-0.03in} M)^\sharp, \|\zeta\|\le 1\}\\
  &=& \lim_\af \sup\{ f(\la \xi-x_\af, P(\zeta)\ra): \zeta\in (H\hspace{-0.03in}\bullet\hspace{-0.03in} M)^\sharp, \|\zeta\|\le 1\}\\\nonumber
 % &=& \lim_\af \sup\{ f(\la P(\zeta), P(\xi-y_\af\ra)): \zeta\in H_A, \|\zeta\|\le 1\}\\
 \hspace{-0.5in} &=& \lim_\af \sup\{ f(\la \xi-y_\af, \zeta\ra): \zeta\in (H\hspace{-0.03in}\bullet\hspace{-0.03in} M)^\sharp, \|\zeta\|\le 1\}\\\nonumber
 &\le &\lim_\af \sup\{ f(\la \xi-y_\af, \zeta\ra): \zeta\in H_M^\sharp, \|\zeta\|\le 1\}=0.
%&=&0.
 \eneq
 %%%%%%%%%%%
 %%%%%%
 \iffalse
 \beq
 |f(\la y, z\ra)-f(\xi, z\ra)|<\ep\rforal z\in X\andeqn f\in {\cal F}.
 \eneq
 Then $Pz=z\in X\subset H^\sim$ and $x:=Py\in H.$ Hence $\|x\|\le \|\xi\|.$  Note that $\la Py, z\ra=\la y, Pz\ra=\la y, z\ra$
 for all $z\in X\subset H.$
 Hence 
 \beq
 |f(\la x, z\ra)-f(\xi,z\ra)|<\ep\rforal z\in H^\sim \andeqn f\in {\cal F}.
 \eneq
 \fi
 %%%%%%%%%%%%
  This proves the lemma.
 \end{proof}

 \begin{thm}\label{S4-mt}
 Let $X$ be a Hilbert space, 
 $A\subset B(X)$ a \SCA\, and $M=\overline{A}^{SOT}$ 
 with $1_M={\rm id}_X,$ 
 and  let $H$ be a Hilbert $A$-module.
 Then the unit ball of $H$ is dense in the unit ball of $(H\hspace{-0.03in}\bullet\hspace{-0.03in} M)^\sharp$ in ${\cal T}_{ws}$ (of $(H\hspace{-0.03in}\bullet\hspace{-0.03in} M)^\sharp$).
 \end{thm}
 
 \begin{proof}
 Let $\xi\in (H\hspace{-0.03in}\bullet\hspace{-0.03in} M)^\sharp$ with $\|\xi\|\le 1.$  It suffices to show that,
 for any $\ep>0,$ any finite subset $Y\subset (H\hspace{-0.03in}\bullet\hspace{-0.03in} M)^\sharp$ and  any finite subset ${\cal F}\subset  M_*,$
 there is $x\in H$ such that
 \beq
\|x\|\le \|\xi\|\andeqn |f(\la \xi-x, y\ra)|<\ep\rforal y\in  (H\hspace{-0.03in}\bullet\hspace{-0.03in} M)^{\sharp}, \|y\|\le 1, \andeqn f\in {\cal F}.
 \eneq
 
 Let us fix $\ep$ and ${\cal F}.$
 Choose an approximate identity $\{E_\lambda\}$ for $K(H).$ 
 It follows that $E_\lambda\nearrow {\rm id}_H.$ Note that
 ${\rm id}_H\in M(K(H)).$ By the last part of Proposition \ref{PPsi}, 
 $\Psi_0({\rm id}_H)={\rm id}_{H\hspace{-0.01in}\bullet\hspace{-0.01in} M}.$  By 3.17 of \cite{Pa}, 
 $F\circ \Psi_0({\rm id}_H)={\rm id}_{(H\hspace{-0.01in}\bullet\hspace{-0.01in} M)^\sharp},$  where $F$ is the map given by Proposition \ref{s2-adp}.  
 Note also that, by Lemma \ref{Lelambda}, $\{\Psi_0(E_\lambda)\}$ is an approximate identity for 
 $K(H\hspace{-0.03in}\bullet\hspace{-0.03in} M).$ 
 In the universal representation of  $K(H\hspace{-0.03in}\bullet\hspace{-0.03in} M),$ 
 $1-\Psi_0(E_\lambda)$ converges to zero in the strong operator topology. 
 Note that $\|1-\Psi_0(E_\lambda)\|\le 1.$ 
 Therefore $(1-\Psi_0(E_\lambda))^*(1-\Psi_0(E_\lambda))$ also converges to zero in the strong operator topology. Hence (since \{$\|(1-\Psi_0(E_\lambda))^*(1-\Psi_0(E_\lambda))\|\}$ is bounded,) it converges to zero in the weak*-topology. 
 % and Theorem \ref{MMT}, 
  %$\Psi({\rm id}_H)={\rm id}_{H^\sim}.$ 
 % Applying Proposition \ref{s2-adp}, 
 By Proposition \ref{s2-adp},
 we have, for all $f\in M_*,$
 \beq
 &&f(\la \xi-F\circ \Psi_0(E_\lambda)(\xi), \xi-F\circ \Psi_0(E_\lambda)(\xi)\ra)\\\label{LLL-100}
 &&=f(\la (1-F\circ \Psi_0(E_\lambda))^*(1-F\circ \Psi_0(E_\lambda))(\xi), \xi\ra)\to 0.
 \eneq
 %\ref{PBL},
 Next let $g$ be a positive normal linear functional in $M_*.$
 Then, for any $y\in  (H\hspace{-0.03in}\bullet\hspace{-0.03in} M)^\sharp$ with $\|y\|\le1,$ 
 %by \eqref{LLL-100},
 \beq
\hspace{-0.5in} |g(\la \xi-F\circ \Psi_0(E_\lambda)(\xi), y \ra)|^2 &\le & g(\la \xi-F\circ \Psi_0(E_\lambda)(\xi), \xi-F\circ \Psi_0(E_\lambda)(\xi)\ra) g(\la y, y\ra)\\
&\le &\|g\|\|y\|^2g(\la \xi-F\circ \Psi_0(E_\lambda)(\xi), \xi-F\circ \Psi_0(E_\lambda)(\xi)\ra).
%%&&\to 0.
 \eneq
 Hence, by \eqref{LLL-100},
 \beq
 \lim_\af  \big(\sup \{|g(\la \xi-F\circ \Psi_0(E_\lambda)(\xi), y \ra)|: y\in (H\hspace{-0.03in}\bullet\hspace{-0.03in} M)^\sharp,\, \|y\|\le 1\}\big)=0.
 \eneq
 It follows that, for any $f\in M_*,$
 \beq
 \lim_\af \big(\sup\{|f(\la \xi-F\circ \Psi_0(E_\lambda)(\xi), y \ra)|: y\in (H\hspace{-0.03in}\bullet\hspace{-0.03in} M)^\sharp,\, \|y\|\le 1\}\big)=0.
 \eneq

 %Combining with \eqref{LLL-100}, 
 Thus we obtain 
$\lambda_0$ such that, 
 for all $\lambda\ge \lambda_0,$ 
 \beq\label{S4-mt-101}
 |f(\la \xi-F\circ\Psi_0(E_\lambda)(\xi), y\ra)|<\ep/2\rforal y\in (H\hspace{-0.03in}\bullet\hspace{-0.03in} M)^{\sharp},\, \|y\|\le 1,\andeqn f\in {\cal F}.
 \eneq
  Let $H_\lambda=\overline{E_\lambda(H)}.$ As in the proof of Theorem \ref{MMT}, $H_\lambda$ is countably 
 generated. Moreover,  by Lemma \ref{ext}, 
 $$
 (H\hspace{-0.03in}\bullet\hspace{-0.03in} M)^{\sharp}=
 %$$
 %H^\sim=
 (H_\lambda\hspace{-0.03in}\bullet\hspace{-0.03in} M)^{\sharp} \oplus ( (H_\lambda\hspace{-0.03in}\bullet\hspace{-0.03in} M)^\sharp)^\perp.
 $$
 Let $P_\lambda: (H\hspace{-0.03in}\bullet\hspace{-0.03in}M)^{\sharp}\to 
 (H_\lambda\hspace{-0.03in}\bullet\hspace{-0.03in} M)^{\sharp}$ 
 %H^\sim\to H_\lambda^\sim$ 
 be the projection.
 Note that $$\Psi_\lambda(E_\lambda)(\xi),\Psi_\lambda(E_\lambda)(y)\in P_\lambda((H\hspace{-0.03in}\bullet\hspace{-0.03in} M)^{\sharp})=(H_\lambda\hspace{-0.03in}\bullet\hspace{-0.03in} M)^{\sharp}$$ 
 for all $y\in (H\hspace{-0.03in}\bullet\hspace{-0.03in}M)^{\sharp}.$
 %In particular, $\Psi_\lambda(E_\lambda)(\xi)\in H_\lambda^\sim$
 %and $P_\lambda(y)\in H_\lambda^\sim$ for all $y\in Y.$ 

 It follows from Lemma \ref{S3-L2} that there is $x\in H_\lambda$ with $\|x\|\le \|\Psi(E_\lambda)(\xi)\|\le \|\xi\|$
 such that
 \beq
 |f(\la \Psi(E_\lambda)(\xi)-x, P_\lambda(y)\ra)|<\ep/2\rforal y\in (H_\lambda\hspace{-0.03in}\bullet\hspace{-0.03in}M)^{\sharp}\andeqn \|y\|\le 1.
 \eneq
 Since $P_\lambda\Psi(E_\lambda)=\Psi(E_\lambda)$  and $x\in H_\lambda,$
 we have, for all  $y\in (H_\lambda\hspace{-0.03in}\bullet\hspace{-0.03in}M)^{\sharp}\andeqn \|y\|\le 1,$
 \beq
|f(\la \Psi(E_\lambda)(\xi)-x, y\ra)|&=&|f(\la P_\lambda\Psi(E_\lambda)(\xi)-P_\lambda(x), y\ra)|\\
&=&|f(\la \Psi(E_\lambda)(\xi)-x, P_\lambda(y)\ra)|<\ep/2. 
 \eneq
 Thus, for all $y\in (H\hspace{-0.03in}\bullet\hspace{-0.03in}M)^{\sharp}$  with 
 $\|y\|\le 1$ and $f\in {\cal F},$ 
 \beq
 |f(\la \xi-x, y\ra)|\le  |f(\la \xi-\Psi(E_\lambda)(\xi), y\ra)|+|f(\la \Psi(E_\lambda)(\xi)-x, y\ra)|
 <\ep/2+\ep/2=\ep.
 \eneq
 \end{proof}
 
 The next two  statements are the main results of this section:
 
 \begin{cor}\label{MC1}
 Let $A$ be a $W^*$-algebra and $H$ be a Hilbert $A$-module.
 Then the unit ball of $H$ is dense in $H^\sharp$ in ${\cal T}_{ws}.$
 \end{cor}
 
 \begin{proof}
 Let $M=A$ and then apply Theorem \ref{S4-mt}.
 \end{proof}
 
 \begin{thm}\label{S6MT}
 Let $A$ be a \CA\, and $H$ be a Hilbert $A$-module.
 Then the unit ball of $H$ is dense in $H^\sim$ in ${\cal T}_{ws}$ 
 (as $H^\sim=(H\hspace{-0.03in}\bullet\hspace{-0.03in} A^{**})^{\sharp}$).
 \end{thm}
 
 \begin{proof}
 We choose the  universal representation $\pi_U$ and its strong operator closure $A''=A^{**}.$
 Then apply Theorem \ref{S4-mt}.
 \end{proof}

 \iffalse
 \begin{cor}
 Let $A$ be a  \CA\, and $H$ be a Hilbert $A$-module.
 Then the unit ball of $H$ is dense in the unit ball of $H^\sim$ in ${\cal T}_w.$
 \end{cor}
 \fi

\end{document}